\documentclass[letterpaper,11pt]{article}
\usepackage[margin=1in]{geometry}
\usepackage{color}
\usepackage{verbatim}
\usepackage{amsmath}
\definecolor{darkred}{rgb}{0.7, 0, 0}
\definecolor{darkgreen}{rgb}{0, 0.7, 0}
\usepackage{amsmath,amsfonts,amssymb,latexsym}
\usepackage{fullpage}
\usepackage{amsmath}
\usepackage{amsthm}
\usepackage{mathtools}
\usepackage{graphicx}
\usepackage{algpseudocode}
\usepackage{placeins}
\usepackage{xcolor}
\usepackage{hyperref}
\usepackage{subcaption}
\usepackage{thmtools}
\usepackage{thm-restate}
\usepackage{tabularx}
\usepackage{dsfont}
\usepackage{placeins}
\usepackage{wrapfig}
\usepackage{todonotes}
\usepackage{enumerate}
\usepackage{float}
\usepackage{xcolor}
\usepackage{indentfirst}
\usepackage{algorithm}
\usepackage{wasysym}
\usepackage{hyperref}
\usepackage{cleveref}
\usepackage{sectsty}
\usepackage{braket}
\usepackage{mathrsfs}
\usepackage{autonum}
\usepackage{bbm}
\usepackage{enumitem}

\DeclareMathOperator{\Ent}{Ent}

\def\N{\mathbb{N}}
\def\R{\mathbb{R}}

\def\cA{\mathcal{A}}
\def\cB{\mathcal{B}}
\def\cC{\mathcal{C}}
\def\cD{\mathcal{D}}
\def\cE{\mathcal{E}}
\def\cF{\mathcal{F}}
\def\cG{\mathcal{G}}
\def\cH{\mathcal{H}}
\def\cK{\mathcal{K}}
\def\cL{\mathcal{L}}
\def\cM{\mathcal{M}}
\def\cN{\mathcal{N}}
\def\cO{\mathcal{O}}

\def\cS{\mathcal{S}}

\def\cW{\mathcal{W}}
\def\cX{\mathcal{X}}
\def\cY{\mathcal{Y}}
\def\cZ{\mathcal{Z}}

\def\fa{\mathfrak{a}}
\def\fB{\mathfrak{B}}

\newcommand{\abs}[1]{\left\lvert#1\right\rvert}
\newcommand{\norm}[1]{\left\|#1\right\|}
\renewcommand{\d}{\mathrm{d}}
\newcommand{\dd}{\,\mathrm{d}}
\newcommand{\E}{\operatorname{\mathbb{E}}}

\newtheorem{theorem}{Theorem}[section]
\newtheorem{lemma}[theorem]{Lemma}
\newtheorem{remark}[theorem]{Remark}
\newtheorem{assumption}[theorem]{Assumption}
\newtheorem{proposition}[theorem]{Proposition}
\newtheorem{corollary}[theorem]{Corollary}

\crefname{theorem}{Theorem}{Theorems}
\crefname{lemma}{Lemma}{Lemmas}
\crefname{remark}{Remark}{Remarks}
\crefname{section}{Section}{Sections}
\crefname{assumption}{Assumption}{Assumptions}
\crefname{proposition}{Proposition}{Propositions}
\crefname{corollary}{Corollary}{Corollaries}
\crefname{figure}{Figure}{Figures}

\begin{document}

\title{Bayesian statistical inverse problems for a coupled Fokker--Planck--Darcy system}

\author{
Grigorios A. Pavliotis\thanks{Department of Mathematics, Imperial College London, London SW7 2AZ, UK, \texttt{g.pavliotis@imperial.ac.uk}.}
\and
Andrew M. Stuart\footnote{Computing and Mathematical Sciences, California Institute of Technology, Pasadena, CA, USA, \texttt{astuart@caltech.edu}.}
\and
Andrea Zanoni\footnote{Centro di Ricerca Matematica Ennio De Giorgi, Scuola Normale Superiore, Pisa, Italy, \texttt{andrea.zanoni@sns.it}.}
}
\date{}

\maketitle

\begin{abstract}
We study the nonparametric statistical inverse problem of recovering the space-dependent permittivity in a coupled Fokker--Planck--Darcy system from discrete, noisy observations of the Fokker--Planck solution. We consider the coupled parabolic-elliptic system on a bounded domain with the physically natural no-flux boundary condition for the Fokker--Planck equation and homogeneous Dirichlet boundary conditions for the Darcy equation. The inverse problem is indirect, since the unknown coefficient enters only through the elliptic equation and is observed only through its effect on the density in the parabolic equation. We first develop the analytical theory of the forward problem required for the statistical analysis, establishing well-posedness and a priori bounds uniform over the admissible set of permittivities. We then derive two stability estimates for the inverse problem, namely a Lipschitz-type forward estimate and a generalized backward estimate. Placing a rescaled Gaussian process prior on the log-shifted permittivity, we show that the posterior contracts around the truth at an explicit polynomial rate in the number of observations, with the posterior mean converging at the same rate. Numerical experiments in one and two dimensions, using preconditioned Crank--Nicolson and ensemble Kalman filter algorithms, complement our theoretical results.
\end{abstract}

%\tableofcontents

\section{Introduction} \label{sec:introduction}

Inverse problems for partial differential equations (PDEs) arise in many applications, ranging from geophysics and petroleum engineering to medical imaging~\cite{TKS20}, biology, chemotaxis~\cite{FBM25}, and semiconductor physics~\cite{TaJ25}. Various approaches have been proposed for studying inverse problems for PDEs, including deterministic regularization techniques~\cite{EHN96} and Bayesian approaches~\cite{KaS05,Stu10}. In particular, nonparametric Bayesian approaches to nonlinear statistical inverse problems~\cite{Nic23} have been applied with great success to several PDEs arising in applications, such as Darcy's law~\cite{PSV22,KSV26}.

The goal of this paper is to study a Bayesian statistical inverse problem for a coupled parabolic-elliptic PDE system, the Fokker--Planck--Darcy PDE 
for the pair $(u,\varphi)$
\begin{equation} \label{eq:PDE_system_intro}
\begin{aligned}
\partial_t u &= \sigma^2 \Delta u + \nabla \cdot \left( u \nabla \varphi \right), && \text{in } \Omega \times (0,T), \\
- \nabla \cdot \left( \alpha \nabla \varphi \right) &= u, && \text{in } \Omega \times (0,T),
\end{aligned}
\end{equation}
together with appropriate boundary conditions. Our goal is to infer the permittivity function $\alpha \colon \Omega \to \R$ from noisy measurements of the density $u$; we will assume that scalar diffusion coefficient $\sigma >0$ is known. The main novelty of the problem that we consider in this paper, particularly with respect to the standard inverse problem for Darcy's law, is that we do not have access to noisy measurements of the solution to the Darcy problem. Rather, the solution of the elliptic PDE is observed indirectly via observations of the solution to the Fokker--Planck equation. Our objective is to show that, under appropriate assumptions regarding the spatiotemporal measurements of the solution to the Fokker--Planck equation, we can demonstrate the convergence of the posterior distribution to the true permittivity for a broad class of prior distributions.

The PDE model~\eqref{eq:PDE_system_intro}, together with the appropriate boundary conditions, appears in semiconductor modeling and the theory of electrolytes~\cite{MRS90,WMZ08,BEL04,Bil92,BWN94}. In the spatially-homogeneous case, the elliptic PDE becomes the Poisson equation. When the medium is heterogeneous, the appropriate model carries a space-dependent permittivity $\alpha$, leading to the Darcy-type equation $-\nabla \cdot (\alpha \nabla \varphi) = u$. The permittivity encodes the material properties of the medium. We note that the inverse problem we consider in this paper is different from the question of estimating the doping profile from boundary measurements, the problem that is considered in~\cite{TaJ25}.

We will adopt a Bayesian nonparametric approach, placing a rescaled Gaussian process prior on a reparametrization of the permittivity and studying the concentration of the resulting posterior distribution as the number of observations grows. This places our work in the framework developed systematically in~\cite{Nic23} for nonlinear statistical inverse problems, in which posterior contraction rates are obtained by combining a contraction result for the observed quantity with a stability estimate for the inverse map. The estimate that will be crucial for our analysis is based on the recent generalized stability estimate for Darcy's law that was obtained in~\cite[Lemma 2.2]{Wan26}.

The same program that we will carry out in this paper for the Fokker--Planck--Darcy system~\eqref{eq:PDE_system_intro} was carried out for the McKean--Vlasov PDE 
\begin{equation}
\partial_t \rho = \Delta \rho + \nabla \cdot \left( \rho \nabla W \ast \rho \right), \quad \text{in } \Omega \times (0,T), 
\end{equation}
on the torus in~\cite{NPR25}, where $W$ denotes the unknown interaction kernel that has to be inferred from noisy space-time measurements of the solution $\rho$ to the McKean--Vlasov PDE. Optimal posterior convergence rates were obtained in that paper for a class of sufficiently regular interaction kernels and for an appropriate choice of initial conditions. In particular, it was shown that it is necessary to consider sufficiently non-smooth initial conditions and restrict observations over a short time interval $[0,t_0]$, where $t_0$ depends on the smoothness of the initial conditions, in order to recover the interaction potential with optimal, close to the parametric case, rates.

In principle, the problem that we consider in this paper can be addressed by reducing it to the one that was studied in~\cite{NPR25}. In particular, we can solve the elliptic PDE in terms of the Green's function, and then substitute the resulting expression into the Fokker--Planck equation to obtain the McKean--Vlasov PDE. The crucial difference is that the convolution term has to be replaced by an integral of the form $\int_\Omega W_\alpha(x,y) u(t,y) \dd y$, where $W_\alpha$ denotes the (Dirichlet) Green's function of the Darcy operator. We can then aim to recover the permittivity from the Green function based on the type of measurements that are considered in~\cite{NPR25}. However, there are several difficulties with this approach. First, the Green's function for a second order uniformly elliptic operator does not satisfy the regularity assumptions on $W$ that were made in~\cite{NPR25}, due to the singularity along the diagonal. Second, even if we could recover the Green's function, it is not clear how we could infer the permittivity unless we had many measurements along the diagonal. In addition, the periodic boundary conditions considered in~\cite{NPR25}, which considerably simplify the PDE analysis, are not appropriate for the applications motivating the problem we consider in the present work. 

In this paper, we consider the inverse problem for the Fokker--Planck--Darcy PDE system~\eqref{eq:PDE_system_intro} with the physically relevant choice of no-flux boundary conditions for the density and homogeneous Dirichlet conditions for the potential~\cite{Bil92, BWN94}. This choice of boundary conditions renders the problem more physically relevant, but also more difficult from a PDE theory perspective. However, the choice of boundary conditions does have a positive mathematical consequence: it enables us to use the stability estimate from~\cite[Lemma 2.2]{Wan26} to bypass the obstruction to the unique solvability of the inverse problem for Darcy's problem that is highlighted in~\cite{Ric81}. Three issues are particularly relevant. First, the no-flux boundary conditions are nonlinear and nonlocal, and the well-posedness of the forward problem is more delicate than in the case of periodic boundary conditions. By extending the results from~\cite{Bil92,BWN94} to the case of non-constant permittivity, we prove that the forward problem is well-posed over a time interval $[0,T]$. Second, the parabolic regularity theory that is straightforward in the case of periodic boundary conditions is much more subtle. In particular, the compatibility conditions that are needed on the initial conditions (see, e.g.,~\cite[Section 7.1]{Eva98},~\cite[Chapter IV]{Wlo87},\cite{Mil99}) in order to have regularity extending all the way to $t=0$ cannot be imposed, since they involve both unknown $u(\cdot,0)$ and $\varphi(\cdot,0)$. We mention, in particular, that only the initial condition $u_0$ for $u$ can be chosen. In fact, $\varphi(\cdot,0)$ needs to be calculated by solving the Poisson equation $- \nabla \cdot \left( \alpha \nabla \varphi(\cdot,0) \right) = u(\cdot,0)$. Hence, the higher order derivatives for $u(\cdot,0)$ and $\varphi(\cdot,0)$ that are needed for the compatibility conditions cannot be satisfied a priori without further analysis of the Darcy PDE at $t=0$. This is the reason why our regularity estimates are stated on time intervals bounded away from $t=0$, and why the observation window is taken with $t_0 > 0$, in contrast to what is done in~\cite{NPR25}. The corresponding constants indeed degenerate as $t_0 \to 0$, but since the window is fixed independently of the number of observations, the contraction rate is unaffected.

On the other hand, our choice of boundary conditions results in the right-hand side of the elliptic PDE in~\eqref{eq:PDE_system_intro} being strictly positive. As a consequence, we can apply~\cite[Lemma 2.2]{Wan26} and obtain a stability estimate crucial in our analysis. However, it is not clear how to extend this approach to the case of periodic boundary conditions, since, in that case, the elliptic PDE takes the form
\begin{equation}
- \nabla \cdot \left( \alpha \nabla \varphi \right) = u - \frac1{\abs{\Omega}} \int_\Omega u(y,\cdot) \dd y,
\end{equation}
as the right-hand side must have zero mean by Fredholm's alternative. Therefore, the right-hand side necessarily changes sign and the stability estimates in~\cite{Nic23,Wan26} cannot be used, the situation that is highlighted in~\cite{Ric81}. Even though numerical experiments for the problem with periodic boundary conditions suggest that it is still possible to recover $\alpha$, just as in the no-flux/Dirichlet boundary conditions case, it is not clear how to extend the analysis to this case.

\paragraph{Our Contributions.} The main contributions of the work are as follows. 

\begin{itemize}[leftmargin=*]

\item We present, in \cref{sec:analysis_PDE}, a detailed analysis of the forward problem for~\eqref{eq:PDE_system_intro} with no-flux and Dirichlet boundary conditions for the Fokker--Planck and Darcy PDEs, respectively. Even though the analysis follows earlier work in, e.g., \cite{Bil92,BWN94}, care must be taken so that this analysis can be adapted to the setting of variable permittivity. In particular, we need to obtain results that hold uniformly over an appropriate space of permittivities, for example when obtaining upper and lower bounds on the density.

\item We present, in \cref{sec:analysis_IP}, a detailed analysis of the inverse problem for the drift-diffusion-Darcy system. We prove a Lipschitz-type forward stability estimate, as well as a backward stability estimate. The generalized stability estimate that we obtain is of the form required by the framework of \cite{Nic23,Wan26}. 

\item Combining these ingredients yields our main result, \cref{thm:contraction_parameter}, which establishes that the posterior distribution contracts around the true permittivity at an explicit, computable rate that depends on a regularity parameter of the solution to the forward problem, and the posterior mean estimator attains the same rate.

\item We present, in \cref{sec:numerics}, numerical experiments in one and two space dimensions that illustrate the theory and compare posterior sampling with an ensemble Kalman approach.

\end{itemize}

\paragraph{Outline.} The remainder of the paper is organized as follows. In \cref{sec:setting}, we present the mathematical model and formulate the corresponding inverse problem. In \cref{sec:analysis_PDE}, we investigate the analytical properties of the system and, after establishing existence and uniqueness of solutions, we derive boundedness and regularity estimates. In \cref{sec:analysis_IP}, we focus on the inverse problem, proving forward and backward stability estimates as well as posterior contraction results. Numerical experiments in dimensions one and two are presented in \cref{sec:numerics}. Finally, in \cref{sec:conclusion}, we summarize the main contributions of the paper and discuss possible directions for future research.

\section{Problem setting} \label{sec:setting}

\subsection{The forward problem} \label{ssec:FP}

Let $\Omega \subset \R^d$, with $d \in \{1,2,3\}$, be a bounded domain with $\cC^\infty$ boundary $\partial\Omega$, and let $T$ be a final time. Consider the following initial-boundary value problem for the pair $(u,\varphi)$
\begin{equation} \label{eq:PDE_system}
\begin{aligned}
\partial_t u &= \sigma^2 \Delta u + \nabla \cdot \left( u \nabla \varphi \right), && \text{in } \Omega \times (0,T), \\
- \nabla \cdot \left( \alpha \nabla \varphi \right) &= u, && \text{in } \Omega \times (0,T),
\end{aligned}
\end{equation}
where $\sigma > 0$ is a constant diffusion coefficient of the Fokker--Planck equation and $\alpha$ is the permittivity function of the Darcy equation. We impose no-flux and homogeneous Dirichlet boundary conditions
\begin{equation} \label{eq:PDE_system_BC}
\begin{aligned}
\sigma^2 \frac{\partial u}{\partial\nu} + u \frac{\partial\varphi}{\partial\nu} &= 0, && \text{on } \partial\Omega \times (0,T), \\
\varphi &= 0, && \text{on } \partial\Omega \times (0,T), \\
\end{aligned}
\end{equation}
where $\nu$ denotes the outward unit normal to $\partial\Omega$, and the initial condition
\begin{equation} \label{eq:PDE_system_IC}
u(\cdot, 0) = u_0, \quad\text{in } \Omega.
\end{equation}
We work under the following conditions for the permittivity $\alpha$ and the initial condition $u_0$.

\begin{assumption} \label{as:alpha_u0}
The functions $\alpha$ and $u_0$ belong to $\cC^\infty(\bar\Omega)$. Moreover, there exist positive constants $a,A,b,B > 0$ such that
\begin{equation}
\begin{aligned}
a \le \alpha(x) \le A &\quad\text{for all } x \in \bar\Omega, \\
b \le u_0(x) \le B &\quad\text{for all } x \in \bar\Omega,
\end{aligned}
\end{equation}
and
\begin{equation}
\norm{u_0}_{L^1(\Omega)} = 1.
\end{equation}
\end{assumption}

In the present setting, the system \eqref{eq:PDE_system} admits a unique solution, a fact that we establish in \cref{sec:analysis_PDE}. Building on this well-posedness result, we turn to the inverse problem of recovering the permittivity $\alpha$ from discrete observations of the solution $u$ to the Fokker–Planck equation. This problem is defined in the next subsection and is the subject of \cref{sec:analysis_IP}. When needed for clarity, we denote the solution $(u,\varphi)$ of \eqref{eq:PDE_system} by $(u(\cdot,\cdot;\alpha),\varphi(\cdot,\cdot;\alpha))$, making explicit the dependence on the permittivity $\alpha$.

\subsection{The inverse problem} \label{ssec:IP}

Let $\alpha_0$ denote the unknown true permittivity parameter of interest, and consider $N$ space-time measurements
\begin{equation} \label{eq:observations}
Y_n = u(x_n, t_n; \alpha_0) + \eta_n, \qquad n = 1, \dots, N,
\end{equation}
where the $\eta_n$ are independent and identically distributed Gaussian random variables with distribution $\cN(0,\gamma^2)$ and known noise level $\gamma$. The design points $(x_n,t_n) \in \Omega \times [0,T]$ are drawn independently from the uniform distribution on the space-time cylinder $\Omega \times [t_0,t_1] \subset \Omega \times [0,T]$ with $0 < t_0 < t_1 \le T$, and are independent of the noise. Note that the observation window $[t_0,t_1]$ can be arbitrarily small, provided it has positive length. We denote the data set by
\begin{equation}
\cD_N = \left\{ \left( x_n, t_n, Y_n \right) \right\}_{n=1}^N,
\end{equation}
and its law by $P_N$. Our goal is to recover $\alpha$ from $\cD_N$ within the admissible set
\begin{equation}
\cA_M^\ell \coloneq \left\{ \alpha \in W^{\ell,\infty}(\Omega) \colon \alpha \ge a, \; \norm{\alpha}_{W^{\ell,\infty}(\Omega)} \le M \right\},
\end{equation}
where $\ell \in \N$, with $\ell \ge 4$, is a regularity parameter, and $M > 0$ is a given constant. We will apply a Bayesian approach in which we seek a probability distribution on $\alpha$ -- see \cref{sec:analysis_IP}. The lower bound and regularity will hold almost surely under the prior and hence the posterior; in contrast, the bound $M$ will depend
on the individual realization, under the prior or the posterior. This causes no difficulty: all the results of \cref{sec:analysis_PDE,ssec:SE} are deterministic and hold for every $\alpha \in \cA_M^\ell$. The role of $M$ in the Bayesian formulation is discussed after \eqref{eq:bayesian_model} and in \cref{rem:ThetaMN}.

\begin{remark} \label{rem:condition_ell}
The requirement that $\ell \ge 4$ arises as follows. The backward stability estimate of \cref{pro:backward_stability_estimate} requires the solution $u$ of \eqref{eq:PDE_system} to satisfy the regularity bound of \cref{as:regularity} for some $\kappa > \max\{2,\, 1+d/2\}$. Here, the condition $\kappa > 2$ makes available the interpolation in time used in the proof of \cref{cor:backward_stability_estimate}, while the condition $\kappa > 1 + d/2$ guaranties the Sobolev embedding $H^\kappa(\Omega) \hookrightarrow W^{1,\infty}(\Omega)$, which is needed in order to apply the elliptic estimates of \cite[Lemma 2.2]{Wan26}. Elliptic regularity for the Darcy equation then requires the permittivity to satisfy $\alpha \in W^{\kappa+1,\infty}(\Omega)$, whence $\ell \ge \kappa+1$. Since $d \le 3$ gives $\max\{2,\,1+d/2\} \le 5/2$, an admissible exponent $\kappa$ exists precisely when $\ell \ge 4$.
\end{remark}

A distinctive feature of this inverse problem is that the parameter is inferred from indirect observations. Indeed, the available data consist of measurements of the density $u$, while the unknown coefficient $\alpha$ enters the model only through the elliptic equation for the potential $\varphi$. Since $u$ depends on $\varphi$ via the coupled system \eqref{eq:PDE_system}, the observations provide only indirect information on $\alpha$.

\section{Properties of the coupled system} \label{sec:analysis_PDE}

In this section, we study the existence and uniqueness of solutions to the coupled system 
defined in \cref{ssec:FP}, and lay the foundation for the statistical analysis of the inverse problem. In \cref{ssec:EU} we study existence and uniqueness, and in
\cref{ssec:BR} we establish a priori bounds and regularity estimates.
All of the following results hold uniformly over $\alpha \in \cA_M^\ell$. In particular, the constants appearing below are independent of $\alpha$ and depend only on $a,b,B,M,\sigma,T,\Omega$, and $u_0$.

\subsection{Existence and uniqueness of solutions} \label{ssec:EU}

The proof of existence and uniqueness follows the arguments in \cite{Bil92}, which we recall here for completeness. The main difference lies in the presence of the permittivity $\alpha$ in the Darcy equation, whereas in \cite{Bil92} the elliptic operator is the Laplacian, corresponding to a Poisson equation without permittivity. Specifically, we employ a fixed point argument in the space $\cX = L^4(0,T; L^2(\Omega))$. We define the operator $\cS \colon \cX \to \cX$ that maps $v \in \cX$ to the function $u \in \cX$ obtained as the solution of the Fokker--Planck equation
\begin{equation} \label{eq:parabolic_PDE_linear}
\begin{aligned}
\partial_t u &= \sigma^2 \Delta u + \nabla \cdot \left( u \nabla \psi \right), && \text{in } \Omega \times (0,T), \\
\sigma^2 \frac{\partial u}{\partial\nu} + u \frac{\partial\psi}{\partial\nu} &= 0, && \text{on } \partial\Omega \times (0,T), \\
u(\cdot, 0) &= u_0, && \text{in } \Omega,
\end{aligned}
\end{equation}
where $\psi$ is the solution of the elliptic problem
\begin{equation} \label{eq:elliptic_PDE_linear}
\begin{aligned}
- \nabla \cdot \left( \alpha \nabla \psi \right) &= v, && \text{in } \Omega \times (0,T), \\
\psi &= 0, && \text{on } \partial\Omega \times (0,T).
\end{aligned}
\end{equation}
In this formulation, the system is decoupled in the sense that the Fokker--Planck equation depends on $\psi$, which is determined by $v$ through the elliptic problem, but not conversely. However a fixed point of $\cS$ delivers a solution of the forward problem defined in \cref{ssec:FP}.

The following result shows that the operator $\cS$ is well-defined, and that, for a suitable radius $\mathsf R$, there exists a ball $\cB_{\mathsf R} \subset \cX$ which is invariant under $\cS$.

\begin{lemma} \label{lem:S_defined_invariant}
Under \cref{as:alpha_u0}, given $v \in \cX$, there exists a unique weak solution $u \in \cX$, defining the operator $\cS \colon \cX \to \cX$ by $\cS(v) = u$. Moreover, there exists $T > 0$ sufficiently small, such that the ball $\cB_{\mathsf R} \subset \cX$ with $\mathsf R = 2 \norm{u_0}_{L^2(\Omega)}$ is invariant under $\cS$, i.e., $\cS(\cB_{\mathsf R}) \subseteq \cB_{\mathsf R}$.
\end{lemma}
\begin{proof}
Let $\widetilde C > 0$ denote a generic constant that may change from line to line. First, a unique solution $\psi \in L^4(0,T;H^2(\Omega))$ of equation \eqref{eq:elliptic_PDE_linear} exists by standard elliptic theory. Moreover, by Sobolev embedding in dimension $d \le 3$ and elliptic regularity, it holds
\begin{equation} \label{eq:bound_nabla_psi_L6}
\norm{\nabla\psi(\cdot,t)}_{L^6(\Omega)} \le \widetilde C\norm{\psi(\cdot,t)}_{H^2(\Omega)} \le \widetilde C \norm{v(\cdot,t)}_{L^2(\Omega)},
\end{equation}
which implies that $\nabla\psi \in L^4(0,T;L^6(\Omega))$. To show the existence of a solution $u$ to the linear parabolic problem \eqref{eq:parabolic_PDE_linear}, we verify the following conditions from \cite[Chapter III]{LSU68} for some $r,q,r',q'$
\begin{equation}
\begin{aligned}
(i) \quad& \abs{\nabla\psi}^2 \in L^r(0,T;L^q(\Omega)) \quad \text{with} \quad \frac1r + \frac{d}{2q} \le 1, \\
(ii) \quad& \frac{\partial\psi}{\partial\nu} \in L^{r'}(0,T;L^{q'}(\partial\Omega)) \quad \text{with} \quad \frac1{r'} + \frac{d-1}{2q'} \le \frac12.
\end{aligned}
\end{equation}
In particular, $(i)$ holds true by taking $r = 2$ and $q = 3$ if $d \le 3$, since
\begin{equation}
\norm{\abs{\nabla\psi}^2}_{L^2(0,T;L^3(\Omega))} = \norm{\nabla\psi}_{L^4(0,T;L^6(\Omega))}^2.
\end{equation}
For $(ii)$, since $\psi(\cdot,t) \in H^2(\Omega)$ and consequently $\nabla\psi(\cdot,t) \in H^1(\Omega)$, we choose $r' = 4$ and $q' = 4$ and use the Sobolev embedding theorem for traces (see, e.g., \cite[Theorem 5.22]{AdF03}) $H^1(\Omega) \hookrightarrow L^4(\partial\Omega)$. Then, we derive a priori bounds for the solution $u$. Testing equation \eqref{eq:parabolic_PDE_linear} with $u$, we have
\begin{equation}
\frac12 \frac{\d}{\d t} \norm{u(\cdot,t)}_{L^2(\Omega)}^2 + \sigma^2 \norm{\nabla u(\cdot,t)}_{L^2(\Omega)}^2 = \int_\Omega u(x,t) \nabla u(x,t) \cdot \nabla \psi(x,t) \dd x \eqcolon Q(t),
\end{equation}
and, using Hölder and Gagliardo--Nirenberg inequalities and bound \eqref{eq:bound_nabla_psi_L6}, we obtain
\begin{equation}
\begin{aligned}
\abs{Q(t)} &\le \widetilde C \norm{u(\cdot,t)}_{L^3(\Omega)} \norm{\nabla u(\cdot,t)}_{L^2(\Omega)} \norm{\nabla\psi(\cdot,t)}_{L^6(\Omega)} \\
&\le \widetilde C \norm{u(\cdot,t)}_{H^1(\Omega)}^{1/2} \norm{u(\cdot,t)}_{L^2(\Omega)}^{1/2} \norm{\nabla u(\cdot,t)}_{L^2(\Omega)} \norm{v(\cdot,t)}_{L^2(\Omega)}.
\end{aligned}
\end{equation}
Applying Young's inequality twice gives
\begin{equation}
\abs{Q(t)} \le \frac{\sigma^2}2 \norm{\nabla u(\cdot,t)}_{L^2(\Omega)}^2 + \widetilde C \left( 1 + \norm{v(\cdot,t)}_{L^2(\Omega)}^4 \right) \norm{u(\cdot,t)}_{L^2(\Omega)}^2,
\end{equation}
which implies
\begin{equation} \label{eq:bound_before_Gronwall}
\frac{\d}{\d t} \norm{u(\cdot,t)}_{L^2(\Omega)}^2 + \sigma^2 \norm{\nabla u(\cdot,t)}_{L^2(\Omega)}^2 \le \widetilde C \left( 1 + \norm{v(\cdot,t)}_{L^2(\Omega)}^4 \right) \norm{u(\cdot,t)}_{L^2(\Omega)}^2.
\end{equation}
Using Grönwall's lemma, then it follows
\begin{equation}
\norm{u(\cdot,t)}_{L^2(\Omega)}^2 \le \norm{u_0}_{L^2(\Omega)}^2 \exp \left( \widetilde C \left( T + \norm{v}_{L^4(0,T;L^2(\Omega))}^4 \right) \right) < \infty,
\end{equation}
which implies $u \in L^\infty(0,T;L^2(\Omega))$ and
\begin{equation} \label{eq:bound_invariant_ball}
\norm{u(\cdot,t)}_{L^4(0,T;L^2(\Omega))} \le T^{1/4} \norm{u_0}_{L^2(\Omega)} \exp \left( \widetilde C \left( T + \norm{v}_{L^4(0,T;L^2(\Omega))}^4 \right) \right).
\end{equation}
Integrating equation \eqref{eq:bound_before_Gronwall}, we also deduce that $u \in L^2(0,T;H^1(\Omega))$. Therefore, by the parabolic theory in \cite[Chapter III]{LSU68}, there exists a unique solution $u \in L^\infty(0,T;L^2(\Omega)) \cap L^2(0,T;H^1(\Omega))$, which implies that $u \in \cX$. Moreover, by equation \eqref{eq:bound_invariant_ball}, for $T$ sufficiently small, the right-hand side can be bounded by $2\norm{u_0}_{L^2(\Omega)}$, yielding the invariance of the ball $\cB_{\mathsf R}$ under $\cS$ and concluding the proof.
\end{proof}

In the next two results, we prove that $\cS$ is continuous and compact, respectively.

\begin{lemma} \label{lem:S_continuous}
Under \cref{as:alpha_u0}, the operator $\cS \colon \cX \to \cX$ is continuous.
\end{lemma}
\begin{proof}
Let $\widetilde C > 0$ denote a generic constant that may change from line to line. Let $v_1,v_2 \in \cX$, $\psi_1,\psi_2$ be the corresponding solutions of \eqref{eq:elliptic_PDE_linear}, and consider $u_1 = \cS(v_1)$ and $u_2 = \cS(v_2)$. From the parabolic problem \eqref{eq:parabolic_PDE_linear}, we deduce
\begin{equation}
\partial_t (u_1 - u_2) = \sigma^2 (\Delta u_1 - \Delta u_2) + \nabla \cdot ((u_1 - u_2) \nabla\psi_1) + \nabla \cdot (u_2(\nabla\psi_1 - \nabla\psi_2)).
\end{equation}
Testing with $u_1 - u_2$, proceeding as in the proof of \cref{lem:S_defined_invariant}, and using that $u_2 \in L^\infty(0,T;L^2(\Omega))$, we obtain
\begin{equation}
\begin{aligned}
\frac{\d}{\d t} \norm{u_1(\cdot,t) - u_2(\cdot,t)}_{L^2(\Omega)}^2 &\le \widetilde C (1 + \norm{v_1(\cdot,t)}_{L^2(\Omega)}^4) \norm{u_1(\cdot,t) - u_2(\cdot,t)}_{L^2(\Omega)}^2 \\
&\quad + \widetilde C \norm{u_2(\cdot,t)}_{H^1(\Omega)} \norm{v_1(\cdot,t) - v_2(\cdot,t)}_{L^2(\Omega)}^2.
\end{aligned}
\end{equation}
Grönwall's lemma with the fact that $u_1(\cdot,0) = u_2(\cdot,0) = u_0$ and Hölder's inequality then give
\begin{equation}
\begin{aligned}
&\norm{u_1(\cdot,t) - u_2(\cdot,t)}_{L^2(\Omega)}^2 \\
&\qquad\le \widetilde C \int_0^t \exp \left( \widetilde C \int_s^t (1 + \norm{v_1(\cdot,r)}_{L^2(\Omega)}^4) \dd r \right) \norm{u_2(\cdot,s)}_{H^1(\Omega)} \norm{v_1(\cdot,s) - v_2(\cdot,s)}_{L^2(\Omega)}^2 \dd s \\
&\qquad\le \widetilde C \exp \left( \widetilde C \left( T + \norm{v}_{L^4(0,T;L^2(\Omega))}^4 \right) \right) \norm{u_2}_{L^2(0,T;H^1(\Omega))} \norm{v_1 - v_2}_{L^4(0,T;L^2(\Omega))}^2,
\end{aligned}
\end{equation}
which implies
\begin{equation}
\begin{aligned}
&\norm{u_1 - u_2}_{L^4(0,T;L^2(\Omega))} \\
&\qquad\le T^{1/4} \widetilde C \exp \left( \widetilde C \left( T + \norm{v}_{L^4(0,T;L^2(\Omega))}^4 \right) \right) \norm{u_2}_{L^2(0,T;H^1(\Omega))}^{1/2} \norm{v_1 - v_2}_{L^4(0,T;L^2(\Omega))},
\end{aligned}
\end{equation}
which shows that $\cS$ is continuous from $\cX$ to $\cX$ and completes the proof.
\end{proof}

\begin{lemma} \label{lem:S_compact}
Under \cref{as:alpha_u0}, the operator $\cS \colon \cX \to \cX$ is compact.
\end{lemma}
\begin{proof}
Let $\widetilde C > 0$ denote a generic constant that may change from line to line. We prove that for any ball $\cB_r \subset \cX$, the image $\cS(\cB_r)$ is precompact in $\cX$. From the proof of \cref{lem:S_defined_invariant}, we know that $u \in L^2(0,T;H^1(\Omega))$. We now show that $\partial_t u = \nabla \cdot (\sigma^2 \nabla u + u \nabla\psi) \in L^2(0,T;H^{-1}(\Omega))$. In fact, for $\phi \in H^1(\Omega)$, proceeding as in the proof of \cref{lem:S_defined_invariant}, we have
\begin{equation}
\begin{aligned}
\abs{\int_\Omega \partial_t u(x,t) \phi(x) \dd x} &\le \sigma^2 \abs{\int_\Omega \nabla u(x,t) \cdot \nabla\phi(x) \dd x} + \abs{\int_\Omega u \nabla\psi(x,t) \cdot \nabla\phi(x) \dd x} \\
&\le \widetilde C \left( \norm{u(\cdot,t)}_{H^1(\Omega)} + \norm{v(\cdot,t)}_{L^2(\Omega)}^2 \right) \norm{\phi}_{H^1(\Omega)}.
\end{aligned}
\end{equation}
Therefore, we obtain
\begin{equation}
\norm{\partial_t u(\cdot,t)}_{H^{-1}(\Omega)} \le \widetilde C \left( \norm{u(\cdot,t)}_{H^1(\Omega)} + \norm{v(\cdot,t)}_{L^2(\Omega)}^2 \right),
\end{equation}
which implies
\begin{equation}
\norm{\partial_t u}_{L^2(0,T;H^{-1}(\Omega))} \le \widetilde C \left( \norm{u}_{L^2(0,T;H^1(\Omega))} + \norm{v}_{L^4(0,T;L^2(\Omega))}^2 \right).
\end{equation}
By the Aubin--Lions compactness lemma, $\cS(\cB_r)$ is precompact in $L^2(0,T;L^2(\Omega))$. Since in addition $u \in L^\infty(0,T;H^1(\Omega))$, this implies precompactness in $\cX$ as well, which gives the desired result.
\end{proof}

The previous lemmas on the operator $\cS$ allow us to conclude that system \eqref{eq:PDE_system} admits a unique solution, as stated in the following theorem.

\begin{theorem} \label{thm:existence_uniqueness}
Under \cref{as:alpha_u0}, there exists $T > 0$ such that the system \eqref{eq:PDE_system} with boundary conditions \eqref{eq:PDE_system_BC} and initial condition \eqref{eq:PDE_system_IC} has a unique solution $(u,\varphi)$ with
\begin{equation} \label{eq:spaces_solution_system}
\begin{aligned}
u &\in L^\infty(0,T; L^2(\Omega)) \cap L^2(0,T; H^1(\Omega)), \\
\partial_t u &\in L^2(0,T; H^{-1}(\Omega)), \\
\varphi &\in L^\infty(0,T; H^2(\Omega)) \cap L^2(0,T; H^3(\Omega)).
\end{aligned}
\end{equation}
Moreover, for all $t \in [0,T]$ it holds
\begin{equation}
u(\cdot,t) \ge 0 \quad a.e. \qquad \text{and} \qquad \norm{u(\cdot,t)}_{L^1(\Omega)} = 1.
\end{equation}
\end{theorem}
\begin{proof}
Combining \cref{lem:S_defined_invariant,lem:S_continuous,lem:S_compact} and applying the Schauder fixed-point theorem, there exists a fixed point $u \in \cX$ such that $u = \cS(u)$, which is a weak solution of system \eqref{eq:PDE_system}. Moreover, the inclusions in \eqref{eq:spaces_solution_system} follow from the proofs of \cref{lem:S_defined_invariant,lem:S_compact} and standard elliptic regularity. We then proceed as in the proof of \cref{lem:S_continuous} with $v_1 = u_1$ and $v_2 = u_2$, deducing
\begin{equation}
\frac{\d}{\d t} \norm{u_1(\cdot,t) - u_2(\cdot,t)}_{L^2(\Omega)}^2 \le \widetilde C (1 + \norm{u_1(\cdot,t)}_{L^2(\Omega)}^4 + \norm{u_2(\cdot,t)}_{H^1(\Omega)}) \norm{u_1(\cdot,t) - u_2(\cdot,t)}_{L^2(\Omega)}^2.
\end{equation}
Since $u_1(\cdot,0) = u_2(\cdot,0) = u_0$, Grönwall's inequality implies uniqueness. For the nonnegativity, testing the Fokker--Planck equation with $u^-(x,t) = - \min(u(x,t), 0)$, we obtain equation \eqref{eq:bound_before_Gronwall} with $v = u = u^-$, and using that $u \in L^\infty(0,T;L^2(\Omega))$, we deduce
\begin{equation}
\frac{\d}{\d t} \norm{u^-(\cdot,t)}_{L^2(\Omega)}^2 \le \widetilde C \norm{u^-(\cdot,t)}_{L^2(\Omega)}^2,
\end{equation}
for a constant $\widetilde C > 0$. Since $u^-(\cdot,0) = 0$ by \cref{as:alpha_u0}, Grönwall's inequality gives $u^-(\cdot,t) = 0$ $a.e.$, i.e., $u(\cdot,t) \ge 0$ $a.e.$, for all $t \in [0,T]$. Finally, the fact that $\norm{u(\cdot,t)}_{L^1(\Omega)} = 1$ for all $t \in [0,T]$ follows directly from the no-flux boundary condition in \eqref{eq:PDE_system_BC}.
\end{proof}

\begin{remark} \label{rem:time_horizon}
\cref{thm:existence_uniqueness} provides a solution on an interval $[0,T^*]$ whose length $T^* > 0$ is determined by \cref{lem:S_defined_invariant}. Inspection of the proof of that lemma shows that $T^*$ depends only on $\norm{u_0}_{L^2(\Omega)}, a, M, \sigma, \Omega, d$ through the elliptic estimate \eqref{eq:bound_nabla_psi_L6}. In particular, $T^*$ is uniform over $\alpha \in \cA_M^\ell$, which is what the analysis of \cref{sec:analysis_IP} requires. Throughout the remainder of the paper we therefore assume that the final time $T$ satisfies $T \le T^*$. This restriction on $T$ is thus uniform over $\cA_M^\ell$ and does not impose any additional restriction on the admissible set. We denote $T^\dagger$ the maximal time of existence of the solution, so that $T \le T^* \le T^\dagger \le \infty$. Whether solutions exist globally in time, i.e., whether $T^\dagger = \infty$, would require studying the free energy of the system, but is left for future work.
\end{remark}

\subsection{Boundedness and regularity of solutions} \label{ssec:BR}

In this section we derive several a priori estimates showing that the pair $(u,\varphi)$ enjoys stronger integrability properties than those obtained in \cref{thm:existence_uniqueness} for the proof of existence and uniqueness. These additional estimates will play a crucial role in \cref{sec:analysis_IP}, where we study the inverse problem. We begin by proving that the solution $u$ to the Fokker--Planck equation is uniformly bounded in both time and space.

\begin{proposition} \label{pro:boundedness_u}
Under \cref{as:alpha_u0}, for all $p \in [1,\infty]$, there exists a constant $C_p > 0$, depending on $p$, such that 
\begin{equation}
\norm{u}_{L^\infty(0,T;L^p(\Omega))} \le C_p.
\end{equation}
\end{proposition}

\begin{proof}
First, assume, without loss of generality, $p\ge2$. In fact, if $p \in [1,2)$, we have $\norm{u}_{L^\infty(0,T;L^p(\Omega))} \lesssim \norm{u}_{L^\infty(0,T;L^2(\Omega))}$ by Hölder's inequality. Then, multiplying the Fokker--Planck equation in \eqref{eq:PDE_system} by $u^{p-1}$ and integrating in $\Omega$ and by parts, we obtain
\begin{equation} \label{eq:testing_up}
\frac1p \frac{\d}{\d t} \norm{u(\cdot,t)}_{L^p(\Omega)}^p + \sigma^2 (p-1) \int_\Omega \abs{\nabla u(x,t)}^2 u(x,t)^{p-2} \dd x = (1-p) \int_\Omega \nabla\varphi(x,t) \cdot \nabla u(x,t) u(x,t)^{p-1} \dd x.
\end{equation}
Applying Cauchy--Schwarz and Young's inequality to the right-hand side, we have
\begin{equation}
\begin{aligned}
\abs{\int_\Omega \nabla\varphi(x,t) \cdot \nabla u(x,t) u^{p-1}(x,t) \dd x}
&\le \frac12 \sigma^2 \int_\Omega \abs{\nabla u(x,t)}^2 u(x,t)^{p-2} \dd x \\
&\quad + \frac1{2\sigma^2} \int_\Omega \abs{\nabla\varphi(x,t)}^2 u(x,t)^p \dd x,
\end{aligned}
\end{equation}
which implies
\begin{equation} \label{eq:inequality_dudt_beforeGronwall}
\begin{aligned}
\frac{\d}{\d t} \norm{u(\cdot,t)}_{L^p(\Omega)}^p + \frac12 \sigma^2 p(p-1) \int_\Omega \abs{\nabla u(x,t)}^2 u(x,t)^{p-2} \dd x &\le \frac{p(p-1)}{2\sigma^2} \int_\Omega \abs{\nabla\varphi(x,t)}^2 u(x,t)^p \dd x \\
&\le \frac{p(p-1)}{2\sigma^2} \norm{\nabla\varphi(\cdot,t)}_{L^\infty(\Omega)}^2 \norm{u(\cdot,t)}_{L^p(\Omega)}^p.
\end{aligned}
\end{equation}
The Sobolev embedding $H^3(\Omega) \hookrightarrow W^{1,\infty}(\Omega)$, which holds true in dimensions $d \le 3$, gives for a constant $\widetilde C$ that may change from line to line
\begin{equation}
\frac{\d}{\d t} \norm{u(\cdot,t)}_{L^p(\Omega)}^p \le \widetilde C p(p-1) \norm{\varphi(\cdot,t)}_{H^3(\Omega)}^2 \norm{u(\cdot,t)}_{L^p(\Omega)}^p,
\end{equation}
and the Grönwall's inequality yields
\begin{equation}
\norm{u(\cdot,t)}_{L^p(\Omega)}^p \le \norm{u_0}_{L^p(\Omega)}^p \exp \left( \widetilde C p(p-1) \int_0^t \norm{\varphi(\cdot,s)}_{H^3(\Omega)}^2 \dd s \right).
\end{equation}
Therefore, due to the assumptions on $u_0$, we obtain for all $p \ge 1$
\begin{equation}
\begin{aligned}
\norm{u}_{L^\infty(0,T;L^p(\Omega))} &\le \norm{u_0}_{L^p(\Omega)} \exp \left( \widetilde C (p-1) \norm{\varphi}_{L^2(0,T;H^3(\Omega))}^2 \right) \\
&\le B \abs{\Omega}^{\frac1p} \exp \left( \widetilde C (p-1) \norm{\varphi}_{L^2(0,T;H^3(\Omega))}^2 \right) \eqcolon C_p,
\end{aligned}
\end{equation}
which concludes the proof for $p \ge 1$ finite. Let us now consider the case $p = \infty$. First, since $u \in L^\infty(0,T;L^p(\Omega))$ for all $p \ge 1$, elliptic regularity yields $\varphi \in L^\infty(0,T;W^{2,p}(\Omega))$. Choosing $p > d$, the Sobolev embedding $W^{2,p}(\Omega) \hookrightarrow W^{1,\infty}(\Omega)$ gives the uniform bound
\begin{equation} \label{eq:boundedness_grad_phi}
\norm{\varphi}_{L^\infty(0,T;W^{1,\infty}(\Omega))} \le \Lambda_1,
\end{equation}
for a positive constant $\Lambda_1 > 0$. Then, setting $w = u^{p/2}$, equation \eqref{eq:inequality_dudt_beforeGronwall} can be rewritten as
\begin{equation}
\frac{\d}{\d t} \norm{w(\cdot,t)}_{L^2(\Omega)}^2 + 2 \sigma^2 \frac{p-1}p \norm{\nabla w(\cdot,t)}_{L^2(\Omega)}^2 \le \frac{p(p-1) \Lambda_1^2}{2\sigma^2} \norm{w(\cdot,t)}_{L^2(\Omega)}^2.
\end{equation}
Let $p_k = 2^k$ for $k \in \N$, $k \neq 0$, and denote $w_k = u^{p_k/2}$. Then, we have
\begin{equation}
\frac{\d}{\d t} \norm{w_k(\cdot,t)}_{L^2(\Omega)}^2 + \sigma^2 \norm{\nabla w_k(\cdot,t)}_{L^2(\Omega)}^2 \le \frac{\Lambda_1^2}{2\sigma^2} p_k^2 \norm{w_k(\cdot,t)}_{L^2(\Omega)}^2.
\end{equation}
To control the right-hand side, Gagliardo--Nirenberg--Sobolev interpolation inequality gives
\begin{equation}
\norm{w_k(\cdot,t)}_{L^2(\Omega)}^2 \le \widetilde C \left( \norm{\nabla w_k(\cdot,t)}_{L^2(\Omega)}^{\frac{2d}{2+d}} \norm{w_k(\cdot,t)}_{L^1(\Omega)}^{\frac4{2+d}} + \norm{w_k(\cdot,t)}_{L^1(\Omega)}^2 \right),
\end{equation}
and Young's inequality implies
\begin{equation}
\frac{\Lambda_1^2}{2\sigma^2} p_k^2 \norm{w_k(\cdot,t)}_{L^2(\Omega)}^2 \le \sigma^2 \norm{\nabla w_k(\cdot,t)}_{L^2(\Omega)}^2 + \widetilde C p_k^{d+2} \norm{w_k(\cdot,t)}_{L^1(\Omega)}^2.
\end{equation}
Therefore, we obtain
\begin{equation}
\frac{\d}{\d t} \norm{w_k(\cdot,t)}_{L^2(\Omega)}^2 \le \widetilde C p_k^{d+2} \norm{w_k(\cdot,t)}_{L^1(\Omega)}^2,
\end{equation}
which, integrating in time, gives
\begin{equation}
\norm{w_k(\cdot,t)}_{L^2(\Omega)}^2 \le \norm{w_k(\cdot,0)}_{L^2(\Omega)}^2 + \widetilde C T p_k^{d+2} \sup_{s \in (0,T)} \norm{w_k(\cdot,s)}_{L^1(\Omega)}^2.
\end{equation}
Recalling that $w_k^2 = u^{p_k}$ and $p_k = 2 p_{k-1}$, we deduce
\begin{equation}
\norm{u(\cdot,t)}_{L^{p_k}(\Omega)}^{p_k} \le \norm{u_0}_{L^{p_k}(\Omega)}^{p_k} + \widetilde C T p_k^{d+2} \sup_{s \in (0,T)} \norm{u(\cdot,s)}_{L^{p_{k-1}(\Omega)}}^{p_k},
\end{equation}
and defining
\begin{equation}
\cM_k \coloneqq \max \left\{ 1, \norm{u_0}_{L^\infty(\Omega)}, \sup_{t \in (0,T)} \norm{u(\cdot,t)}_{L^{p_k}(\Omega)} \right\},
\end{equation}
we obtain
\begin{equation}
\cM_k^{p_k} \le \widetilde C 2^{(d+2)k} \cM_{k-1}^{p_k},
\end{equation}
which implies
\begin{equation}
\cM_k \le \widetilde C^{2^{-k}} 2^{(d+2) k 2^{-k}} \cM_{k-1}.
\end{equation}
Iterating this relation, we find
\begin{equation}
\cM_k \le \cM_0 \prod_{j=1}^k \widetilde C^{2^{-j}} 2^{(d+2) j 2^{-j}} = \cM_0 \widetilde C^{\sum_{j=1}^k 2^{-j}} 2^{(d+2) \sum_{j=1}^k j 2^{-j}},
\end{equation}
where, since $\norm{u(\cdot,t)}_{L^1(\Omega)} = 1$ for all $t$ by \cref{thm:existence_uniqueness}, the starting value is explicit,
\begin{equation}
\cM_0 = \max \left\{ 1, \norm{u_0}_{L^\infty(\Omega)} \right\} \le \max\{1, B\}.
\end{equation}
Since the series at the exponent satisfies
\begin{equation}
 \sum_{j=0}^\infty 2^{-j} = 2 \qquad \text{and} \qquad \sum_{j=0}^\infty j 2^{-j} = 2,
\end{equation}
taking the limit as $k \to \infty$ yields
\begin{equation}
\norm{u}_{L^\infty(0,T;L^\infty(\Omega))} \le \lim_{k\to\infty} \cM_k \le \max\{1,B\} \, \widetilde C^2 2^{2(d+2)} \eqcolon C_\infty,
\end{equation}
which concludes the proof.
\end{proof}

The uniform bounds established in \cref{pro:boundedness_u} immediately yield higher regularity for the solution $\varphi$ to the Darcy equation, as shown in the next result.

\begin{corollary} \label{cor:boundedness_grad_Delta_phi}
Under \cref{as:alpha_u0}, there exists two constants $\Lambda_1, \Lambda_2 > 0$ such that 
\begin{equation}
\norm{\varphi}_{L^\infty(0,T;W^{1,\infty}(\Omega))} \le \Lambda_1 \qquad \text{and} \qquad \norm{\Delta\varphi}_{L^\infty(0,T;L^\infty(\Omega))} \le \Lambda_2 \coloneq \frac1a (M \Lambda_1 + C_\infty).
\end{equation}
\end{corollary}
\begin{proof}
The first estimate is given in equation \eqref{eq:boundedness_grad_phi} in the proof of \cref{pro:boundedness_u}. Then, from the Darcy equation in \eqref{eq:PDE_system}, we deduce
\begin{equation}
\Delta\varphi = - \nabla\log\alpha \cdot \nabla\varphi - \frac1\alpha u,
\end{equation}
and the second bound follows from the first bound and \cref{pro:boundedness_u}.
\end{proof}

In addition to the upper bounds derived in \cref{pro:boundedness_u}, the next result shows that $u$ is also bounded from below by a positive constant, uniformly in both time and space. The argument follows the same strategy as the proofs of \cite[Lemma 5.1]{LaL23}, \cite[Lemma 2.18]{PaZ25}, and \cite[Theorem 5]{NPR25}, which in turn are based on the proof of \cite[Theorem 2]{GLM24}. The main difference is that, since our setting involves a bounded domain $\Omega$, we must additionally employ the Skorokhod equation.

\begin{proposition} \label{pro:boundedness_u_below}
Under \cref{as:alpha_u0}, there exists a constant $c_0 > 0$ such that
\begin{equation}
u(x,t) \ge c_0 \coloneq b e^{- \Lambda_2 T}, \quad \text{for all } x \in \Omega \text{ and } t \in [0,T],
\end{equation}
where $\Lambda_2$ is defined in \cref{cor:boundedness_grad_Delta_phi}.
\end{proposition}
\begin{proof}
Let us first rewrite the Fokker--Planck equation in \eqref{eq:PDE_system} as
\begin{equation} \label{eq:PDE_expanded}
\partial_t u = \sigma^2 \Delta u + \nabla\varphi \cdot \nabla u + \Delta\varphi u,
\end{equation}
and consider the unique strong solution $\cZ_t$ of the normally reflected diffusion on $\bar\Omega$ \cite{LiS84}
\begin{equation} \label{eq:reflected_SDE}
\d\cZ_t = \nabla\varphi(\cZ_t, t - s) \dd t + \sqrt{2\sigma^2} \dd W_t - \nu(\cZ_t) \dd L_t,
\end{equation}
with initial condition $\cZ_0 = x \in \Omega$, where $W_t$ is a standard $d$-dimensional Brownian motion, $\nu$ is the outward unit normal to $\partial\Omega$, and $L_t$ is the continuous, non-decreasing boundary local time process. The local time $L_t$ starts at $0$ and increases only when $\cZ_t \in \partial\Omega$, satisfying
\begin{equation}
\int_0^t \mathbbm{1}_{\{\cZ_s \in \partial\Omega\}} \mathrm{d} L_s = L_t.
\end{equation}
Fixing $t \in (0,T]$, we define the backward-in-time function $w(y, s) = u(y, t-s)$ for $y \in \bar\Omega$ and $s \in [0,t-\varepsilon]$ with $\varepsilon \in (0,t)$. Applying Itô--Skorokhod formula for reflected diffusions (see, e.g., \cite[Chapter 7]{IkW89}) to $w(\cZ_s, s)$ yields
\begin{equation} \label{eq:ito_reflected}
\begin{aligned}
\d w(\cZ_s, s) &= \left( \partial_s w(\cZ_s, s) + \nabla\varphi(\cZ_s, t-s) \cdot \nabla w(\cZ_s, s) + \sigma^2 \Delta w(\cZ_s, s) \right) \dd s \\
&\quad + \sqrt{2\sigma^2} \nabla w(\cZ_s, s) \cdot \d W_s - \frac{\partial w}{\partial\nu}(\cZ_s, s) \dd L_s,
\end{aligned}
\end{equation}
which due to equation \eqref{eq:PDE_expanded} and the no-flux boundary conditions \eqref{eq:PDE_system_BC}, implies
\begin{equation}
\d w(\cZ_s, s) = - \Delta\varphi(\cZ_s, t-s) w(\cZ_s, s) \dd s + \sqrt{2\sigma^2} \nabla w(\cZ_s, s) \cdot \d W_s + \frac1{\sigma^2} w(\cZ_s, s) \frac{\partial \varphi}{\partial\nu}(\cZ_s, t-s) \dd L_s.
\end{equation}
Note that, in order to apply the Itô--Skorokhod formula to $w$, we need $u \in \mathcal C^{2,1}(\bar\Omega \times [\varepsilon, T])$, which is guaranteed by \cite{BWN94} and standard parabolic regularity theory \cite[Chapter III]{LSU68}.
Let now $\cY_s = \cE_s w(\cZ_s,s)$, where
\begin{equation}
\cE_s \coloneq \exp \left( \int_0^s \Delta\varphi(\cZ_r, t-r) \dd r - \frac1{\sigma^2} \int_0^s \frac{\partial \varphi}{\partial\nu}(\cZ_r, t-r) \dd L_r \right),
\end{equation}
and notice that $\cY_s \ge 0$ and
\begin{equation}
\d \cY_s = \sqrt{2\sigma^2} \cE_s \nabla w(\cZ_s,s) \cdot \dd W_s.
\end{equation}
Therefore, $\cY_s$ is a supermartingale and it follows
\begin{equation}
u(x,t) = \cE_0 w(\cZ_0, 0) = \cY_0 \ge \E[\cY_{t-\varepsilon}] = \E[\cE_{t-\varepsilon} w(\cZ_{t-\varepsilon}, t-\varepsilon)] = \E[\cE_{t-\varepsilon} u(\cZ_{t-\varepsilon},\varepsilon)],
\end{equation}
which implies
\begin{equation} \label{eq:supermatringale_bound}
u(x,t) \ge \inf_{x \in \Omega} u(x, \varepsilon) \E[\cE_{t-\varepsilon}].
\end{equation}
Then, the second estimate in \cref{cor:boundedness_grad_Delta_phi} gives
\begin{equation} \label{eq:bound_int_Deltaphi}
\int_0^s \Delta\varphi(\cZ_r, t-r) \dd r \ge - \Lambda_2 s \ge - \Lambda_2 T.
\end{equation}
Moreover, from the Darcy equation $-\nabla \cdot (\alpha \nabla\varphi) = u\ge 0$ with homogeneous Dirichlet boundary conditions $\varphi = 0$ on $\partial\Omega$, the maximum principle yields $\varphi \ge 0$ in $\Omega$. Hence, applying Hopf's lemma, the outward normal derivative satisfies
\begin{equation} \label{eq:dphi_sign}
\frac{\partial \varphi}{\partial\nu}(\cZ_r, t-r) \le 0.
\end{equation}
Since $L_r$ is nondecreasing, which yields $dL_r \ge 0$, equations \eqref{eq:bound_int_Deltaphi} and \eqref{eq:dphi_sign} give
\begin{equation}
\cE_{t-\varepsilon} \ge e^{- \Lambda_2 T},
\end{equation}
which, together with equation \eqref{eq:supermatringale_bound}, implies
\begin{equation}
u(x,t) \ge \inf_{x \in \Omega} u(x, \varepsilon) e^{- \Lambda_2 T}.
\end{equation}
Finally, by the smoothness of the initial condition \eqref{eq:PDE_system_IC}, the no-flux boundary condition \eqref{eq:PDE_system_BC}, and the boundedness of the coefficients due to \cref{cor:boundedness_grad_Delta_phi}, the solution of the parabolic PDE satisfies $u \in \cC^0(\bar\Omega \times [0,T])$ without requiring any compatibility condition. Therefore, $u(\cdot,\varepsilon) \to u_0$ uniformly as $\varepsilon\to0$, and by \cref{as:alpha_u0} we deduce
\begin{equation}
u(x,t) \ge \lim_{\varepsilon\to0} \inf_{x \in \Omega} u(x, \varepsilon) e^{- \Lambda_2 T} = \inf_{x \in \Omega} u_0(x) e^{- \Lambda_2 T} \ge b e^{- \Lambda_2 T} \eqcolon c_0,
\end{equation}
which concludes the proof.
\end{proof}

\section{Analysis of the inverse problem} \label{sec:analysis_IP}

In this section, we provide an analysis of the statistical inverse problem of recovering $\alpha$ from $\cD_N$, as defined in \cref{ssec:IP}. First, we define and discuss the problem setting. Then we derive a stability estimate in \cref{ssec:SE}, and we finally deploy this to obtain a posterior contraction property in \cref{ssec:PCR}.  To ensure that the permittivity $\alpha$ is automatically bounded from below by $a$ without any truncation, we reparametrize via
\begin{equation} \label{eq:f2alpha}
\alpha = a + e^f,
\end{equation}
where $f \colon \Omega \to \R$ is a real-valued function to be given a Gaussian process prior. The true log-shifted permittivity is $f_0 = \log(\alpha_0-a)$, and the admissible set for $f$ reads
\begin{equation}
\cF_{M'}^\ell \coloneq \left\{ f \in W^{\ell,\infty}(\Omega) \colon \norm{f}_{W^{\ell,\infty}(\Omega)} \le M' \right\},
\end{equation}
for some $M' > 0$ different from $M$ in $\cA_M^\ell$, such that if $f \in \cF_{M'}^\ell$, then $a + e^f \in \cA_M^\ell$. Specifically, we remark that such $M'$ always exists, and we introduce the operator $\fa \colon \cF_{M'}^\ell \to \cA_M^\ell$ that maps $f$ to $\alpha$ through equation \eqref{eq:f2alpha}. Throughout, we assume without loss of generality that $M \le M'$, which may always be arranged by enlarging $M'$.

It is convenient at this point to record the Bayesian problem in compact form. Collecting the $N$ observations $Y \coloneq (Y_1,\dots,Y_N)$ defined by \eqref{eq:observations} into a single vector, writing $\iota \coloneq \{(x_n,t_n)\}_{n=1}^N$ for the design points and $\eta \coloneq (\eta_1,\dots,\eta_N)$ for the noise, the data are generated as
\begin{equation} \label{eq:bayesian_model}
Y = \cG(f;\iota) + \eta, \qquad
\cG(f;\iota) \coloneq \left( u(x_n,t_n; \fa(f)) \right)_{n=1}^N,
\end{equation}
and the object of interest is the conditional law of $f$ given $(Y,\iota)$, that is, the posterior distribution $\Pi_N(\cdot \mid \cD_N)$ defined in \eqref{eq:posterior}. The forward map $\cG(\cdot\,;\iota)$ is nonlinear, and the unknown $f$ enters it only through the elliptic equation in \eqref{eq:PDE_system}. This formulation makes clear the sense in which the admissible sets $\cA_M^\ell$ and $\cF_{M'}^\ell$ are used in what follows, a point worth making explicit because the two constraints defining them are of a different nature under a Gaussian prior. The regularity $f \in W^{\ell,\infty}(\Omega)$ holds almost surely, by \cref{as:prior} and the Sobolev embedding $H^\beta(\Omega) \hookrightarrow W^{\ell,\infty}(\Omega)$ guaranteed by \eqref{eq:conditions_regularity}; the bound $\norm{f}_{W^{\ell,\infty}(\Omega)} \le M'$, by contrast, does not hold, since a centered Gaussian measure on an infinite-dimensional space charges every norm ball with probability strictly less than one. The bound is therefore a property of the individual realization, and all the estimates of \cref{sec:analysis_PDE,ssec:SE} are correspondingly deterministic: they hold pathwise for every $\alpha \in \cA_M^\ell$, and are applied to a posterior draw $\fa(f)$ only when that draw satisfies the bound. Accordingly, the contraction statements of \cref{ssec:PCR} are formulated on the sets $\Theta_{M,N}$ of \eqref{eq:def_Theta_MN}, on which the bound holds by construction; that this entails no loss is discussed in \cref{rem:ThetaMN}.

We now define the Gaussian process prior for the function $f$. Let $s,\beta > 0$ be regularity parameters with 
\begin{equation} \label{eq:conditions_regularity}
\beta > \ell + \frac{d}2 \qquad \text{and} \qquad s > \beta + \frac{d}2 > \ell + d,
\end{equation}
and consider a centred Gaussian probability measure $\Pi$ on a separable Banach space $\cW$ satisfying
\begin{equation} \label{eq:embedding_W}
\cW \hookrightarrow H^\beta(\Omega),
\end{equation}
whose reproducing kernel Hilbert space (RKHS) $(\cH,\|\cdot\|_\cH)$ satisfies
\begin{equation} \label{eq:embedding_H}
\cH \hookrightarrow H^s(\Omega).
\end{equation}
Note that this is always possible due to \cite[Theorem B.1.3]{Nic23}. We summarize these conditions in \cref{as:prior} below.

\begin{assumption} \label{as:prior}
The centred Gaussian measure $\Pi$ on $\cW$ has RKHS $\cH$ with embeddings \eqref{eq:embedding_W} and \eqref{eq:embedding_H}, where $\beta$ and $s$ satisfy condition \eqref{eq:conditions_regularity}. Moreover, the true permittivity $\alpha_0$ is such that $f_0 = \log(\alpha_0 - a) \in \cH$.
\end{assumption}

Notice that, due to \cref{as:prior}, the Sobolev embedding $H^\beta \hookrightarrow W^{\ell,\infty}(\Omega)$ implies that the support $\cW$ is included in $W^{\ell,\infty}(\Omega)$. We then follow the approach in \cite[Section 2]{Nic23} and rescale the base prior to introduce additional regularization. We define the regularized prior $\Pi_N$ as the law of the Gaussian process $f/(\sqrt N \delta_N)$ with $f \sim \Pi$ and
\begin{equation} \label{eq:delta_N}
\delta_N = N^{- \frac{s}{2s + d}}.
\end{equation}
The posterior distribution of the regularized prior $\Pi_N$ given the data $\cD_N$ is obtained using the standard Bayes' formula and is given by
\begin{equation} \label{eq:posterior}
\d\Pi_N(f | \cD_N) \propto e^{\cL_N(f)} \d\Pi_N(f),
\end{equation}
where the log-likelihood is
\begin{equation}
\cL_N(f) = - \frac1{2\gamma^2} \sum_{n=1}^N \left( Y_n - u(x_n,t_n; \fa(f)) \right)^2.
\end{equation}
The posterior distribution $\pi_N(\cdot | \cD_N)$ on the permittivity $\alpha$ is then the pushforward measure of $\Pi_N(\cdot | \cD_N)$ through $\fa$, i.e., $\pi_N(\cdot | \cD_N) = \fa_\# \Pi_N(\cdot | \cD_N)$. Moreover, let us define the posterior mean estimator $\widehat\alpha_N$ as
\begin{equation} \label{eq:posterior_mean_estimator}
\widehat\alpha_N = \E^{\Pi_N} \left[ \fa(f) | \cD_N \right];
\end{equation}
this can be approximated, for example, using Markov Chain Monte Carlo or sequential Monte Carlo algorithms.

For the subsequent analysis, we impose an additional regularity assumption on the solution of the system \eqref{eq:PDE_system}.

\begin{assumption} \label{as:regularity}
There exists $\kappa \in \R$ with $\max \{ 2, 1+d/2 \} < \kappa \le \ell - 1$ such that, for every $\bar t \in (0,T)$, there is a constant $\Gamma_\kappa = \Gamma_\kappa(\bar t) > 0$ for which, for all $\alpha \in \cA_M^\ell$, the solution u satisfies the higher spatial and temporal regularity estimate
\begin{equation} \label{eq:spatial_estimates}
\norm{u}_{L^\infty(\bar t,T;H^\kappa(\Omega))} + \norm{u}_{H^{\kappa/2}(\bar t,T;L^2(\Omega))} \le \Gamma_\kappa.
\end{equation} 
\end{assumption}

\begin{remark} \label{rem:optimal_regularity}
It can be shown that \cref{as:regularity} holds for some $2 < \kappa < 3$, although this is not expected to be the optimal regularity. Indeed, following the arguments in \cite{BWN94} and applying standard parabolic regularity theory, as presented for example in \cite[Chapter III]{LSU68}, one obtains $u \in \cC^{2 + \mathsf r, 1 + \mathsf r/2}(\Omega \times (0,T])$ for some $\mathsf r \in (0,1)$. Consequently, by Sobolev embedding, for every $\bar t > 0$ and all $t \in [\bar t,T]$, the spatial regularity satisfies $u(\cdot,t) \in H^\kappa(\Omega)$ and in time it holds $u \in H^{\kappa/2}(\bar t,T;L^2(\Omega))$, for some $\kappa \in (2, 2 + \mathsf r)$. Thus, \cref{as:regularity} is certainly satisfied in dimensions $d \le 2$, whereas for $d = 3$ we would need $\mathsf r \in (1/2,1)$. Nevertheless, we believe that it should be possible to establish the estimate with $\kappa = \ell - 1$, corresponding to the highest regularity allowed by the regularity of the diffusion coefficient in the Darcy equation. Proving this sharper regularity estimate is left for future work.
\end{remark}

\begin{remark} \label{rem:compatibility_conditions}
The estimate in \cref{as:regularity} is stated on an interval bounded away from the initial time, and this restriction is not merely technical. Regularity of solutions of parabolic initial-boundary value problems up to the parabolic corner $\partial\Omega \times \{0\}$ requires compatibility conditions between the initial condition and the boundary condition \cite[Chapter IV]{LSU68}. In the present setting these conditions fail for generic initial data because the initial potential is not free but depends on $u_0$ through the elliptic problem, so that $\varphi(\cdot,0)$ is a nonlocal functional of $u_0$. Already the first order compatibility condition is therefore a nonlocal constraint on $u_0$. Moreover, since $u_0 \ge b > 0$ implies $\varphi(\cdot,0) \ge 0$ and hence, by the Hopf lemma, $\partial\varphi(\cdot, 0) / \partial\nu < 0$ on $\partial\Omega$, it forces the definite sign $\partial u_0 / \partial\nu > 0$ everywhere on $\partial\Omega$, which fails, e.g., for any admissible $u_0$ that is constant near the boundary. Consequently, one cannot expect the bound of \cref{as:regularity} to hold uniformly up to $\bar t = 0$, and $\Gamma_\kappa$ is expected to blow up as $\bar t \to 0$. This is the reason the observation window $[t_0, t_1]$ is taken with $t_0 > 0$. We remark that this issue would be invisible in the periodic setting, where there is no boundary and hence no compatibility condition.
\end{remark}

\subsection{Stability estimates} \label{ssec:SE}

Before deriving contraction properties for the posterior distribution, we establish stability estimates for the system \eqref{eq:PDE_system}. In particular, we first prove a Lipschitz-type stability estimate showing continuous dependence of the Fokker--Planck solution $u$ on $\alpha$ in $L^2$. We then derive a complementary estimate in the reverse direction, providing control of the coefficient $\alpha$ in terms of the solution $u$. Both estimates hold uniformly in $\cA_M^\ell$, and the constants appearing in the results depend only on $a,b,B,M,\sigma,T,\Omega$, and $u_0$.

The following relative entropy estimate is the key ingredient in the proof of the forward stability estimate. An analogous result in the whole space $\R^d$ was established in \cite[Theorem 1.1]{BRS16} and subsequently strengthened in \cite[Lemma 3.1]{LaL23}. Here, we prove the corresponding estimate for a bounded domain $\Omega$ with no-flux boundary conditions.

\begin{lemma} \label{lem:entropy_estimate}
Let $\alpha_1, \alpha_2 \in \cA^\ell_M$ and let $(u_1, \varphi_1)$, $(u_2,\varphi_2)$ be the corresponding solutions of the system \eqref{eq:PDE_system} with boundary conditions \eqref{eq:PDE_system_BC} and the common initial condition $u_0$ in \eqref{eq:PDE_system_IC}. Then, for every $t \in [0,T]$, it holds
\begin{equation}
\Ent(u_1(\cdot, t) | u_2(\cdot, t)) \le \frac1{4\sigma^2} \int_0^t \int_\Omega \abs{\nabla\varphi_1(x,s) - \nabla\varphi_2(x,s)}^2 u_1(x,s) \dd x \dd s,
\end{equation}
where the relative entropy is defined for two probability density functions $v_1,v_2$ as
\begin{equation}
\Ent(v_1|v_2) = \int_\Omega v_1(x) \log \left( \frac{v_1(x)}{v_2(x)} \right) \dd x.
\end{equation}
\end{lemma}
\begin{proof}
By definition of relative entropy and the conservation of mass in \cref{thm:existence_uniqueness}, we have
\begin{equation}
\frac{\d}{\d t} \Ent(u_1(\cdot, t) | u_2(\cdot, t)) = \int_\Omega \partial_t u_1(x,t) \log \left( \frac{u_1(x,t)}{u_2(x,t)} \right) \dd x - \int_\Omega \partial_t u_2(x,t) \frac{u_1(x,t)}{u_2(x,t)} \dd x,
\end{equation}
which, using the Fokker--Planck equation in \eqref{eq:PDE_system} and integrating by parts, gives
\begin{equation} \label{eq:derivative_relative_entropy}
\begin{aligned}
\frac{\d}{\d t} \Ent(u_1(\cdot, t) | u_2(\cdot, t)) &= - \int_\Omega \left( \sigma^2 \nabla u_1(x,t) + u_1(x,t) \nabla\varphi_1(x,t) \right) \cdot \nabla \log \left( \frac{u_1(x,t)}{u_2(x,t)} \right) \\
&\quad + \int_\Omega \left( \sigma^2 \nabla u_2(x,t) + u_2(x,t) \nabla\varphi_2(x,t) \right) \cdot \nabla \left( \frac{u_1(x,t)}{u_2(x,t)} \right),
\end{aligned}
\end{equation}
where there are no boundary terms due to the no-flux boundary conditions \cref{eq:PDE_system_BC}. Let us now define
\begin{equation}
\xi(x,t) \coloneq \nabla \log \left( \frac{u_1(x,t)}{u_2(x,t)} \right) = \frac{\nabla u_1(x,t)}{u_1(x,t)} - \frac{\nabla u_2(x,t)}{u_2(x,t)},
\end{equation}
and note that
\begin{equation}
\nabla \left( \frac{u_1(x,t)}{u_2(x,t)} \right) = \frac{u_1(x,t)}{u_2(x,t)} \xi(x,t).
\end{equation}
Therefore, equation \eqref{eq:derivative_relative_entropy} can be rewritten as
\begin{equation}
\begin{aligned}
\frac{\d}{\d t} \Ent(u_1(\cdot, t) | u_2(\cdot, t)) &= - \sigma^2 \int_\Omega \left( \nabla u_1(x,t) - \frac{u_1(x,t)}{u_2(x,t)} \nabla u_2(x,t) \right) \cdot \xi(x,t) \dd x \\
&\quad - \int_\Omega \left( \nabla\varphi_1(x,t) - \nabla\varphi_2(x,t) \right) \cdot \xi(x,t) u_1(x,t) \dd x,
\end{aligned}
\end{equation}
which implies
\begin{equation}
\frac{\d}{\d t} \Ent(u_1(\cdot, t) | u_2(\cdot, t)) = - \sigma^2 \int_{\Omega} \abs{\xi(x,t)}^2 u_1(x,t) \dd x - \int_\Omega \left( \nabla\varphi_1(x,t) - \nabla\varphi_2(x,t) \right) \cdot \xi(x,t) u_1(x,t) \dd x.
\end{equation}
Completing the square, we obtain
\begin{equation}
\begin{aligned}
\frac{\d}{\d t} \Ent(u_1(\cdot, t) | u_2(\cdot, t)) &= - \int_\Omega \abs{\sigma \xi(x,t) + \frac1{2\sigma} \left( \nabla\varphi_1(x,t) - \nabla\varphi_2(x,t) \right)}^2 u_1(x,t) \dd x \\
&\quad + \frac1{4\sigma^2} \int_\Omega \abs{\nabla\varphi_1(x,t) - \nabla\varphi_2(x,t)}^2 u_1(x,t) \dd x,
\end{aligned}
\end{equation}
which yields
\begin{equation}
\frac{\d}{\d t} \Ent(u_1(\cdot, t) | u_2(\cdot, t)) \le \frac1{4\sigma^2} \int_\Omega \abs{\nabla\varphi_1(x,t) - \nabla\varphi_2(x,t)}^2 u_1(x,t) \dd x.
\end{equation}
Since $u_1(\cdot,0) = u_2(\cdot,0) = u_0$, and therefore $ \Ent(u_1(\cdot, 0) | u_2(\cdot, 0)) = 0$, integrating in time gives the desired estimate.
\end{proof}

The forward stability estimate is presented below.

\begin{proposition} \label{pro:Lipschitz}
Let $\alpha_1, \alpha_2 \in \cA_M^\ell$ and let $u_1, u_2$ be the corresponding solutions of the Fokker--Planck equation in the system \eqref{eq:PDE_system}. Under \cref{as:alpha_u0}, there exists a constant $L > 0$ such that
\begin{equation}
\norm{u_1 - u_2}_{L^2(0,T;L^2(\Omega))} \le L \norm{\alpha_1 - \alpha_2}_{L^2(\Omega)}.
\end{equation}
\end{proposition}
\begin{proof}
Let $\varphi_1, \varphi_2$ be the solutions of the Darcy equations corresponding to $u_1, u_2$, respectively. We have
\begin{equation}
- \nabla \cdot (\alpha_1 (\nabla\varphi_1 - \nabla\varphi_2)) = (u_1 - u_2) + \nabla \cdot ((\alpha_1 - \alpha_2) \nabla\varphi_2),
\end{equation}
and testing it with $(\varphi_1 - \varphi_2)$ we obtain
\begin{equation}
\begin{aligned}
\int_\Omega \alpha_1(x) \abs{\nabla\varphi_1(x,t) - \nabla\varphi_2(x,t)}^2 \dd x &= \int_\Omega (u_1(x,t) - u_2(x,t))(\varphi_1(x,t) - \varphi_2(x,t)) \dd x \\
&\quad - \int_\Omega (\alpha_1(x) - \alpha_2(x)) \nabla\varphi_2(x,t) \cdot (\nabla\varphi_1(x,t) - \nabla\varphi_2(x,t)) \dd x.
\end{aligned}
\end{equation}
By \cref{as:alpha_u0}, elliptic estimates, and Poincaré inequality, we deduce
\begin{equation}
a \norm{\nabla\varphi_1(\cdot,t) - \nabla\varphi_2(\cdot,t)}_{L^2(\Omega)} \le \frac1{\sqrt{\lambda_1}} \norm{u_1(\cdot,t) - u_2(\cdot,t)}_{L^2(\Omega)} + \norm{\nabla\varphi_2(\cdot,t)}_{L^\infty(\Omega)} \norm{\alpha_1 - \alpha_2}_{L^2(\Omega)},
\end{equation}
where $\lambda_1$ denotes the minimal eigenvalue of the negative Laplacian in $H^1_0(\Omega)$, and which implies
\begin{equation} \label{eq:bound_difference_nabla_phi}
\norm{\nabla\varphi_1(\cdot,t) - \nabla\varphi_2(\cdot,t)}_{L^2(\Omega)} \le \frac1{a\sqrt{\lambda_1}} \norm{u_1(\cdot,t) - u_2(\cdot,t)}_{L^2(\Omega)} + \frac{\Lambda_1}{a} \norm{\alpha_1 - \alpha_2}_{L^2(\Omega)}.
\end{equation}
We now use the relative entropy estimate in \cref{lem:entropy_estimate}, which gives
\begin{equation}
\Ent(u_1(\cdot,t)|u_2(\cdot,t)) \le \frac1{4\sigma^2} \int_0^t \int_\Omega \abs{\nabla\varphi_1(\cdot,s) - \nabla\varphi_2(\cdot,s)}^2 u_1(x,s) \dd x \dd s.
\end{equation}
Using estimate \eqref{eq:bound_difference_nabla_phi} and the boundedness of $u_1$ from \cref{pro:boundedness_u} with $p = \infty$, we obtain
\begin{equation}
\begin{aligned}
\Ent(u_1(\cdot,t)|u_2(\cdot,t)) &\le \frac{C_\infty}{4\sigma^2} \int_0^t \norm{\nabla\varphi_1(\cdot,s) - \nabla\varphi_2(\cdot,s)}_{L^2(\Omega)}^2 \dd s \\
&\le \frac{C_\infty \Lambda_1^2}{2\sigma^2 a^2} t \norm{\alpha_1 - \alpha_2}_{L^2(\Omega)}^2 + \frac{C_\infty}{2\sigma^2 a^2 \lambda_1} \int_0^t \norm{u_1(\cdot,s) - u_2(\cdot,s)}_{L^2(\Omega)}^2 \dd s.
\end{aligned}
\end{equation}
Using the argument at the beginning of the proof of \cite[Lemma 3]{NPR25}, we can bound the $L^2(\Omega)$ distance with the relative entropy as
\begin{equation} \label{eq:bound_L2_entropy}
\norm{u_1(\cdot, s) - u_2(\cdot, s)}_{L^2(\Omega)}^2 \le 4 C_\infty \Ent(u_1(\cdot, s) | u_2(\cdot, s)),
\end{equation}
which implies
\begin{equation}
\Ent(u_1(\cdot,t)|u_2(\cdot,t)) \le \frac{C_\infty \Lambda_1^2}{2\sigma^2 a^2} t \norm{\alpha_1 - \alpha_2}_{L^2(\Omega)}^2 + \frac{2 C_\infty^2}{\sigma^2 a^2 \lambda_1} \int_0^t \Ent(u_1(\cdot, s) | u_2(\cdot, s)) \dd s.
\end{equation}
Applying Grönwall's inequality, we deduce
\begin{equation}
\Ent(u_1(\cdot,t)|u_2(\cdot,t)) \le \frac{C_\infty \Lambda_1^2}{2 \sigma^2 a^2} t \exp \left( \frac{2 C_\infty^2}{\sigma^2 a^2 \lambda_1} t \right) \norm{\alpha_1 - \alpha_2}_{L^2(\Omega)}^2,
\end{equation}
and using again bound \eqref{eq:bound_L2_entropy}, we obtain
\begin{equation} \label{eq:bound_difference_u}
\norm{u_1(\cdot, t) - u_2(\cdot, t)}_{L^2(\Omega)}^2 \le \frac{2 C_\infty^2 \Lambda_1^2}{\sigma^2 a^2} t \exp \left( \frac{2 C_\infty^2}{\sigma^2 a^2 \lambda_1} t \right) \norm{\alpha_1 - \alpha_2}_{L^2(\Omega)}^2.
\end{equation}
Finally, setting
\begin{equation}
L \coloneq \frac{\sqrt2 C_\infty \Lambda_1}{\sigma a} \left( \int_0^T t \exp \left( \frac{2 C_\infty^2}{\sigma^2 a^2 \lambda_1} t \right) \dd t \right)^{1/2} < \infty,
\end{equation}
the desired result follows by integrating in time equation \eqref{eq:bound_difference_u}.
\end{proof}

In the next result, we prove the backward stability estimate, showing that differences in the coefficient $\alpha$ can be controlled by differences in the solution $u$.

\begin{proposition} \label{pro:backward_stability_estimate}
Let $\alpha_1, \alpha_2 \in \cA_M^\ell$ and let $(u_1,\varphi_1), (u_2,\varphi_2)$ be the corresponding solutions of the system \eqref{eq:PDE_system}. Under \cref{as:alpha_u0,as:regularity}, there exists a constant $K > 0$ such that
\begin{equation}
\norm{\alpha_1 - \alpha_2}_{L^2(\Omega)} \le \frac{K}{\sqrt{t_1 - t_0}} \left( \norm{u_1 - u_2}_{L^2(t_0,t_1;H^2(\Omega))} + \norm{\partial_t u_1 - \partial_t u_2}_{L^2(t_0,t_1;L^2(\Omega))} \right),
\end{equation}
where $K$ depends on $M$, on $\Lambda_1, \Lambda_2$ defined in \cref{cor:boundedness_grad_Delta_phi}, and on the constant $\Gamma_\kappa$ of \cref{as:regularity}, hence also on $t_0$.
\end{proposition}
\begin{proof}
Let $\widetilde K > 0$ denote a generic constant that may change from line to line. Let us first rewrite the Fokker--Planck equation for $u$ as an elliptic PDE for $\varphi$ with $t$ fixed
\begin{equation} \label{eq:FokkerPlanck_elliptic}
- \nabla \cdot \left( u \nabla \varphi \right) = \sigma^2 \Delta u - \partial_t u.
\end{equation}
Considering equation \eqref{eq:FokkerPlanck_elliptic} for both $(u_1,\varphi_1)$ and $(u_2,\varphi_2)$, we get 
\begin{equation}
\begin{aligned}
- \nabla \cdot (u_1 \nabla (\varphi_1 - \varphi_2)) &= \sigma^2 \Delta(u_1 - u_2) - \partial_t (u_1 - u_2) + \nabla \cdot \left( (u_1 - u_2) \nabla \varphi_2 \right) \\
&= \sigma^2 \Delta(u_1 - u_2) - \partial_t (u_1 - u_2) + \nabla (u_1 - u_2) \cdot \nabla \varphi_2 + (u_1 - u_2) \Delta \varphi_2,
\end{aligned}
\end{equation}
which, by elliptic estimates, implies
\begin{equation} \label{eq:bound_difference_phi}
\begin{aligned}
&\norm{\varphi_1(\cdot,t) - \varphi_2(\cdot,t)}_{H^2(\Omega)} \le \widetilde K \left( \norm{\Delta u_1(\cdot,t) - \Delta u_2(\cdot,t)}_{L^2(\Omega)} + \norm{\partial_t u_1(\cdot,t) - \partial_t u_2(\cdot,t)}_{L^2(\Omega)} \right. \\
&\qquad\quad\left. + \norm{\nabla u_1(\cdot,t) - \nabla u_2(\cdot,t)}_{L^2(\Omega)} \norm{\nabla\varphi_2(\cdot,t)}_{L^\infty(\Omega)} + \norm{u_1(\cdot,t) - u_2(\cdot,t)}_{L^2(\Omega)} \norm{\Delta\varphi_2(\cdot,t)}_{L^\infty(\Omega)} \right) \\
&\qquad\le \widetilde K (1 + \norm{\varphi_2(\cdot,t)}_{W^{1,\infty}(\Omega)} + \norm{\Delta\varphi_2(\cdot,t)}_{L^\infty(\Omega)}) \norm{u_1(\cdot,t) - u_2(\cdot,t)}_{H^2(\Omega)} \\
&\qquad\quad + \widetilde K \norm{\partial_t u_1(\cdot,t) - \partial_t u_2(\cdot,t)}_{L^2(\Omega)} \\
&\qquad\le \widetilde K (1 + \Lambda_1 + \Lambda_2) \norm{u_1(\cdot,t) - u_2(\cdot,t)}_{H^2(\Omega)} + \widetilde K \norm{\partial_t u_1(\cdot,t) - \partial_t u_2(\cdot,t)}_{L^2(\Omega)}.
\end{aligned}
\end{equation}
Note that elliptic estimates can be employed above since $u$ is bounded from below by a positive constant due to \cref{pro:boundedness_u_below}, and $u(\cdot,t) \in H^\kappa(\Omega) \hookrightarrow W^{1,\infty}(\Omega)$ by \cref{as:regularity} and Sobolev embedding in $d \le 3$. We now apply \cite[Lemma 2.2]{Wan26} and first verify that its hypotheses are satisfied. By \cref{as:alpha_u0,as:regularity} and elliptic regularity, $u_1(\cdot,t), u_2(\cdot,t) \in H^\kappa(\Omega)$ and $\varphi_1(\cdot,t), \varphi_2(\cdot,t) \in H^{\kappa+2}(\Omega)$ with $\kappa > \max \{ 2, 1 + d/2 \}$, since $\alpha_1, \alpha_2 \in W^{\ell,\infty}$ with $\ell \ge \kappa + 1 \ge 4$. Hence, all regularity assumptions are satisfied. The positivity constraints on the diffusion coefficient and the right-hand side follow from \cref{as:alpha_u0,pro:boundedness_u_below}, respectively. Moreover, note that the minus sign in front of the divergence in the elliptic operator does not affect the estimate, since one may set $\widetilde\varphi \coloneq - \varphi$ and apply \cite[Lemma 2.2]{Wan26} to $\widetilde\varphi$. It then follows from \cite[Lemma 2.2]{Wan26} and the admissible set for $\alpha_2$ that
\begin{equation}
\begin{aligned}
\norm{\alpha_1 - \alpha_2}_{L^2(\Omega)} &\le \widetilde K \left( \norm{\alpha_2}_{W^{1,\infty}(\Omega)} \norm{\varphi_1(\cdot,t) - \varphi_2(\cdot,t)}_{H^2(\Omega)} + \norm{u_1(\cdot,t) - u_2(\cdot,t)}_{L^2(\Omega)} \right) \\
&\le \widetilde K \left( M \norm{\varphi_1(\cdot,t) - \varphi_2(\cdot,t)}_{H^2(\Omega)} + \norm{u_1(\cdot,t) - u_2(\cdot,t)}_{L^2(\Omega)} \right).
\end{aligned}
\end{equation}
Finally, using the estimate \eqref{eq:bound_difference_phi} and \cref{cor:boundedness_grad_Delta_phi} for $\varphi_2$, we obtain
\begin{equation}
\begin{aligned}
\norm{\alpha_1 - \alpha_2}_{L^2(\Omega)} \le \frac1{\sqrt2} K \left( \norm{u_1(\cdot,t) - u_2(\cdot,t)}_{H^2(\Omega)} + \norm{\partial_t u_1(\cdot,t) - \partial_t u_2(\cdot,t)}_{L^2(\Omega)} \right),
\end{aligned}
\end{equation}
where $K$ depends on $M$, on $\Lambda_1,\Lambda_2$ and, through the bound $\norm{\varphi_1(\cdot,t)}_{\cC^2(\Omega)} \lesssim \norm{\varphi_1(\cdot,t)}_{H^{\kappa+2}(\Omega)} \lesssim \norm{u_1(\cdot,t)}_{H^\kappa(\Omega)} \le \Gamma_\kappa$, also on $\Gamma_\kappa$. We stress that $\norm{\varphi_1(\cdot,t)}_{\cC^2(\Omega)}$ is not controlled by $\Lambda_1$ and $\Lambda_2$ alone, since a bound on $\Delta\varphi_1(\cdot,t)$ in $L^\infty(\Omega)$ does not bound the full Hessian. The desired result now follows by integrating in the time interval $[t_0,t_1]$.
\end{proof}

\begin{remark}
The proof of \cref{pro:backward_stability_estimate}, specifically the application of \cite[Lemma 2.2]{Wan26}, is the only place where the requirement $\kappa\le\ell-1$ in \cref{as:regularity} is needed, through the coefficient regularity of the Darcy equation. This requirement is the reason for the restriction $\ell\ge4$ recorded in \cref{rem:condition_ell}.
\end{remark}

While the first term on the right-hand side of the backward stability estimate above is well-defined by \cref{as:regularity}, the second term requires further justification. We address this below and derive a better version of the backward stability estimate.

\begin{corollary} \label{cor:backward_stability_estimate}
Under the assumptions of \cref{pro:backward_stability_estimate}, there exists a constant $\mathcal K > 0$ such that
\begin{equation}
\norm{\alpha_1 - \alpha_2}_{L^2(\Omega)} \le \cK \norm{u_1 - u_2}_{L^2(t_0,t_1;L^2(\Omega))}^{1-\frac2\kappa}.
\end{equation}
\end{corollary}
\begin{proof}
Let $\widetilde C > 0$ denote a generic constant that may change from line to line. Applying a Gagliardo--Nirenberg interpolation inequality, we obtain
\begin{equation}
\norm{u_1(\cdot,t) - u_2(\cdot,t)}_{H^2(\Omega)} \le \widetilde C \norm{u_1(\cdot,t) - u_2(\cdot,t)}_{L^2(\Omega)}^{1-\frac2\kappa} \norm{u_1(\cdot,t) - u_2(\cdot,t)}_{H^\kappa(\Omega)}^{\frac2\kappa},
\end{equation}
which, due to \cref{as:regularity} with $\bar t = t_0$ and Hölder's inequality, implies
\begin{equation}
\begin{aligned}
\norm{u_1 - u_2}_{L^2(t_0, t_1; H^2(\Omega))} &\le \widetilde C \left( \int_{t_0}^{t_1} \norm{u_1(\cdot,t) - u_2(\cdot,t)}_{L^2(\Omega)}^{2-\frac4\kappa} \dd t \right)^{1/2} \\
&\le \widetilde C (t_1 - t_0)^{\frac1\kappa} \norm{u_1 - u_2}_{L^2(t_0,t_1;L^2(\Omega))}^{1-\frac2\kappa},
\end{aligned}
\end{equation}
Proceeding similarly in time, applying the Gagliardo--Nirenberg interpolation inequality, and using \cref{as:regularity}, we deduce
\begin{equation}
\begin{aligned}
 \norm{\partial_t u_1 - \partial_t u_2}_{L^2(t_0,t_1;L^2(\Omega))} &\le \norm{u_1 - u_2}_{H^1(t_0,t_1;L^2(\Omega))} \\
 &\le \widetilde C \norm{u_1 - u_2}_{L^2(t_0,t_1;L^2(\Omega))}^{1-\frac2\kappa} \norm{u_1 - u_2}_{H^{\kappa/2}(t_0,t_1;L^2(\Omega))}^{\frac2\kappa} \\
 &\le \widetilde C \norm{u_1 - u_2}_{L^2(t_0,t_1;L^2(\Omega))}^{1-\frac2\kappa},
 \end{aligned}
\end{equation}
Therefore, applying \cref{pro:backward_stability_estimate} and denoting $\cK = K \widetilde C / \sqrt{t_1 - t_0}$ gives the desired result.
\end{proof}

\begin{remark} \label{rem:K_large}
The constant $\cK$ in \cref{cor:backward_stability_estimate} can potentially be very large. In fact, its dominant contribution comes from the multiplier argument in the proof of \cite[Lemma 2.2]{Wan26}, which, in turn, depends on the lower bound $c_0$ in \cref{pro:boundedness_u_below}. This leads to a constant that is doubly exponential in the final time $T$. Moreover, $\cK$ degenerates both as $t_0 \to 0$ and as $t_1 - t_0 \to 0$, corresponding, respectively, to approaching $t = 0$ and to considering an observation window that is too small. However, neither of these issues affects the contraction rate in \cref{thm:contraction_parameter} below, since $t_0, t_1$, and $T$ are fixed. Nevertheless, the stability estimate should be regarded as qualitative rather than quantitative.
\end{remark}

\subsection{Posterior contraction rates} \label{ssec:PCR}

We now state and prove the main results of this work, namely contraction rates for the posterior distribution $\Pi_N$. We first derive below a contraction rate for the solution $u$ of the Fokker--Planck equation.

\begin{theorem} \label{thm:contraction_solution}
Under \cref{as:alpha_u0,as:prior}, if $R,M > 0$ are sufficiently large, it holds
\begin{equation}
\lim_{N\to\infty} \Pi_N \left( \left. f \in \Theta_{M,N} \colon \norm{u(\cdot,\cdot;\fa(f)) - u(\cdot,\cdot; \alpha_0)}_{L^2(0,T;L^2(\Omega))} > R \delta_N \right| \cD_N \right) = 0, \quad \text{in } P_N,
\end{equation}
where
\begin{equation} \label{eq:def_Theta_MN}
\Theta_{M,N} = \left\{ f = f_1 + f_2 \in \cW \colon \norm{f_1}_{L^2(\Omega)} \le M \delta_N,\; \norm{f_2}_{\cH} \le M,\; \norm{f}_{W^{\ell,\infty}(\Omega)} \le M \right\}.
\end{equation}
\end{theorem}
\begin{proof}
Using the framework of \cite{Nic23} with regularization space $W^{\ell,\infty}(\Omega)$ and forward smoothing parameter of the forward operator equal to zero, we verify \cite[Condition 2.1.1]{Nic23} in order to apply \cite[Theorem 2.2.2]{Nic23}. We also note that, due to \cref{as:prior}, $\cH \subset H^s(\Omega)$ and $f_0 \in \cH \cap W^{\ell,\infty}(\Omega) = \cH$, and therefore the additional assumptions of \cite[Theorem 2.2.2]{Nic23} are satisfied. Using \cref{pro:boundedness_u} with $p = \infty$, we have that
\begin{equation}
\sup_{f \in \cF_{M'}^\ell} \sup_{x \in \Omega,\; t \in [0,T]} \abs{u(x,t;\fa(f))} \le C_\infty,
\end{equation}
which verifies \cite[equation (2.3)]{Nic23}. Moreover, applying \cref{pro:Lipschitz}, for all $f_1, f_2 \in \cF_{M'}^\ell$, we deduce that
\begin{equation}
\norm{u(\cdot,\cdot;e^{f_1}) - u(\cdot,\cdot;e^{f_2})}_{L^2(0,T;L^2(\Omega))} \le L \norm{e^{f_1} - e^{f_2}}_{L^2(\Omega)} \le Le^{M'} \norm{f_1 - f_2}_{L^2(\Omega)},
\end{equation}
which corresponds to \cite[equation (2.4)]{Nic23}. Therefore, \cite[Condition 2.1.1]{Nic23} is satisfied and \cite[Theorem 2.2.2]{Nic23} gives the desired result.
\end{proof}

\begin{remark} \label{rem:ThetaMN}
The statements of \cref{thm:contraction_solution,thm:contraction_parameter} are formulated on the sets $\Theta_{M,N}$, which impose the bound $\norm{f}_{W^{\ell,\infty}(\Omega)} \le M$ that, as noted in the discussion preceding \eqref{eq:bayesian_model}, holds realisation by realisation rather than almost surely. This restriction is asymptotically immaterial. Indeed, \cite[Theorem 2.2.2]{Nic23} also provides, for $M$ sufficiently large, that
\begin{equation} \label{eq:ThetaMN_mass}
\Pi_N \left( \left. \Theta_{M,N}^c \right| \cD_N \right) \longrightarrow 0 \quad \text{in } P_N \text{ as } N \to \infty,
\end{equation}
so that the posterior concentrates precisely on the set where the deterministic estimates of \cref{sec:analysis_PDE,ssec:SE} are available; this is the role of the requirement that $M$ be sufficiently large in the two theorems. The distinction does, however, matter for the posterior mean estimator \eqref{eq:posterior_mean_estimator}, which averages over the whole posterior and not only over $\Theta_{M,N}$; the contribution of $\Theta_{M,N}^c$ is controlled by \eqref{eq:ThetaMN_mass} together with \cite[Theorem 2.3.2]{Nic23}, which is what allows \eqref{eq:estimate_posterior_mean_estimator} to be deduced from \eqref{eq:contraction_parameter}.
\end{remark}

Finally, in the next result, we establish the corresponding posterior contraction rate for the unknown permittivity $\alpha$.

\begin{theorem} \label{thm:contraction_parameter}
Under \cref{as:alpha_u0,as:prior,as:regularity}, if $R,M > 0$ are sufficiently large, it holds
\begin{equation} \label{eq:contraction_parameter}
\lim_{N\to\infty} \Pi_N \left( \left. f \in \Theta_{M,N} \colon \norm{\fa(f) - \alpha_0}_{L^2(\Omega)} > R \delta_N^{1 - \frac2\kappa} \right| \cD_N \right) = 0, \quad \text{in } P_N,
\end{equation}
where $\Theta_{M,N}$ is defined in equation \eqref{eq:def_Theta_MN} and $\kappa$ is given in \cref{as:regularity}. Moreover, the posterior mean estimator satisfies
\begin{equation} \label{eq:estimate_posterior_mean_estimator}
\norm{\widehat\alpha_N - \alpha_0}_{L^2(\Omega)} = \cO_{P_N} \left( \delta_N^{1 - \frac2\kappa} \right) = \cO_{P_N} \left( N^{-\mu} \right), \quad \text{with} \quad \mu = \frac{s(\kappa - 2)}{\kappa (2s + d)},
\end{equation}
where $\widehat\alpha_N$ is defined in equation \eqref{eq:posterior_mean_estimator}, and $\cO_{P_N}$ denotes the standard stochastic big-$\cO$ notation with respect to the probability measure $P_N$.
\end{theorem}
\begin{proof}
Let $R' > 0$ and consider the set
\begin{equation}
\fB_{M,N} \coloneq \left\{ f \in \Theta_{M,N} \colon \mathfrak U(\fa(f),\alpha_0) \le R'\delta_N \right\},
\end{equation}
where $\mathfrak U$ is defined for $\alpha_1,\alpha_2 \in \cA_M^\ell$ as
\begin{equation}
\mathfrak U(\alpha_1,\alpha_2) \coloneq \norm{u(\cdot,\cdot;\alpha_1) - u(\cdot,\cdot;\alpha_2)}_{L^2(t_0,t_1;L^2(\Omega))}.
\end{equation}
In the set $\fB_{M,N}$, we can apply \cref{cor:backward_stability_estimate} with $\alpha_1 = \fa(f)$ and $\alpha_2 = \alpha_0$. In particular, we have
\begin{equation}
\Pi_N \left( \left. \fB_{M,N} \right| \cD_N \right) \le \Pi_N \left( \left. f \in \Theta_{M,N} \colon \norm{\fa(f) - \alpha_0}_{L^2(\Omega)} \le \cK (R')^{1 - \frac2\kappa} \delta_N^{1 - \frac2\kappa} \right| \cD_N \right),
\end{equation}
which, by choosing $R' = (R/\cK)^{\kappa/(\kappa-2)}$, implies
\begin{equation}
\Pi_N \left( \left. f \in \Theta_{M,N} \colon \norm{\fa(f) - \alpha_0}_{L^2(\Omega)} > R \delta_N^{1 - \frac2\kappa} \right| \right) \le \Pi_N(\fB_{M,N}^c | \cD_N).
\end{equation}
The claim in \eqref{eq:contraction_parameter} then follows by applying \cref{thm:contraction_solution} and choosing $R'$, and hence $R$, sufficiently large. Finally, estimate \eqref{eq:estimate_posterior_mean_estimator} for the posterior mean estimator is an immediate consequence of \eqref{eq:contraction_parameter} together with \cite[Theorem 2.3.2]{Nic23} and \eqref{eq:delta_N}.
\end{proof}

\section{Numerical experiments} \label{sec:numerics}

In this section, we present results obtained by numerical solution of the inverse problem of estimating the permittivity coefficient $\alpha$ of the Darcy equation from discrete observations of the solution to the Fokker--Planck equation, as defined in \cref{ssec:IP}. We consider the cases $d=1$ and $d=2$, for which the regularity condition in \cref{as:regularity} is satisfied, as discussed in \cref{rem:optimal_regularity}. We set the diffusion coefficient to $\sigma=1$ and the final time to $T=1$. We assign a Matern-like Gaussian prior distribution $\mathcal N(0, \Sigma)$, on $\R^d$, to the true log-shifted permittivity $f$ (with $a = 1$), where
\begin{equation} \label{eq:covariance}
\Sigma = \varsigma^2 (\tau^2 I - \Delta)^{- \theta},
\end{equation}
and then consider the restriction of $f$ to the domain $\Omega$.
From this prior distribution, we sample the true unknown parameter $f_0$, which is approximated using a truncated Karhunen--Loève expansion with $50$ modes. We remark that, although a sample from $\mathcal N(0,\Sigma)$ belongs to its RKHS with probability zero, \cref{as:prior} is satisfied, since a finite Karhunen--Loève sum lies in $\mathcal H$. Note that (see, e.g., \cite[Section 2.1]{LRL11}) the Gaussian distribution $\mathcal N(0,\Sigma)$ corresponds to a Matérn kernel with smoothness parameter $\theta - d/2$, correlation length $1/\tau$, and standard deviation
\begin{equation}
\sqrt{\frac{\varsigma^2 \Gamma(\theta - d/2)}{(4\pi)^{d/2} \Gamma(\theta) \tau^{2\theta - d}}}.
\end{equation}
The standard deviation of the observational noise is assumed to be $\gamma = 0.002$. The forward system \eqref{eq:PDE_system} is discretized using P1 finite elements in space and an implicit Euler scheme in time. Samples from the posterior distribution $\pi_N(\cdot | \mathcal D_N)$ are obtained via the pCN algorithm. The numerical experiments are implemented in Julia, using the \texttt{EnsembleKalmanProcesses} package for prior sampling and \texttt{Gridap} for the finite element discretization. In \cref{ssec:DNO} we numerically study dependence on the number of observations, in a one-dimensional setting; in \cref{ssec:2D} we describe numerical experiments in two dimensions.

\subsection{Dependence on the number of observations} \label{ssec:DNO}

\begin{figure}
\begin{center}
\begin{tabular}{ccc}
\includegraphics{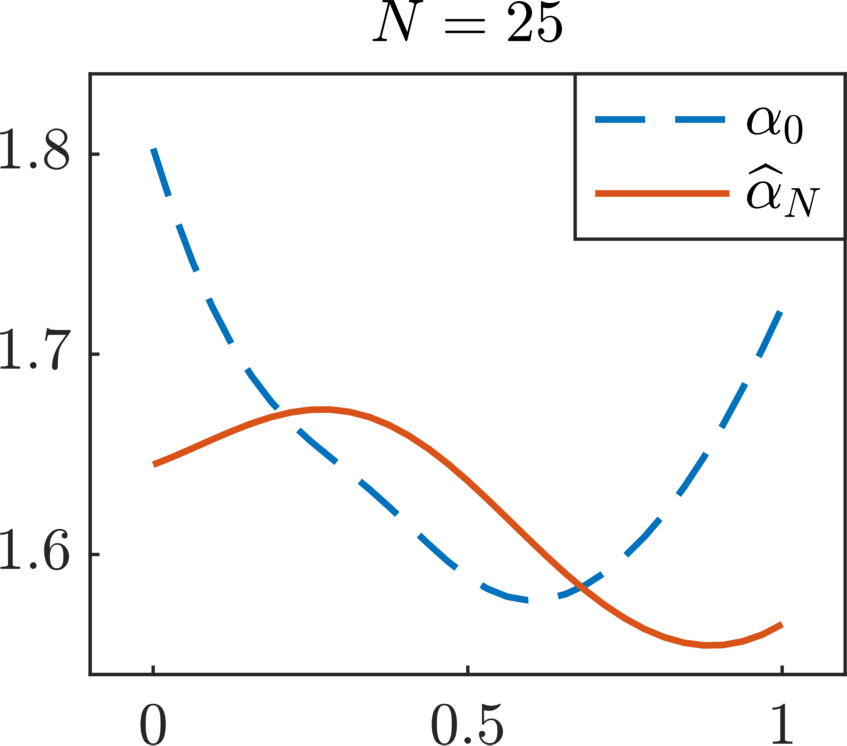} &
\includegraphics{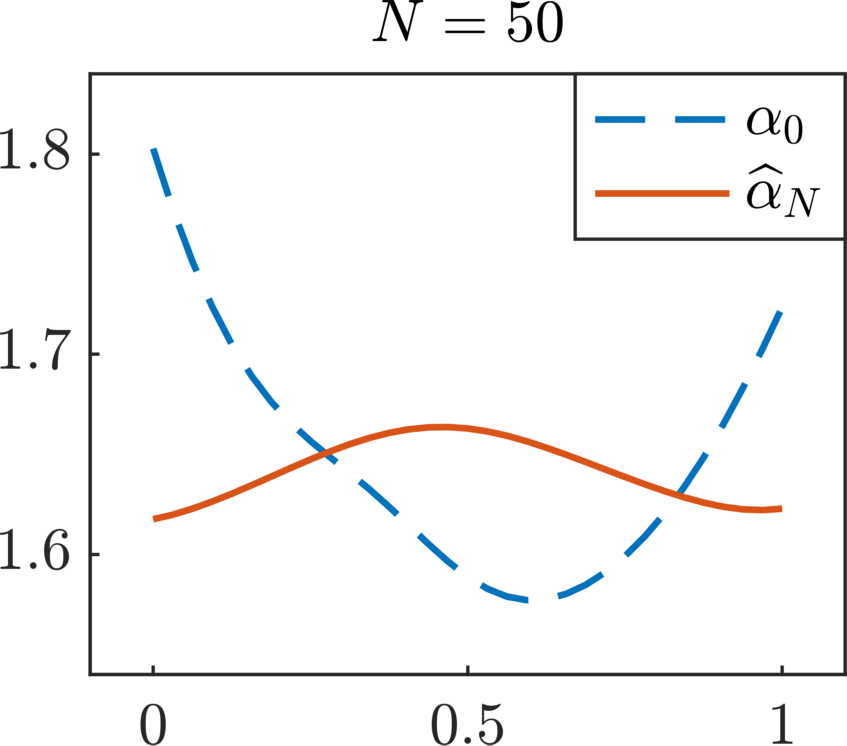} &
\includegraphics{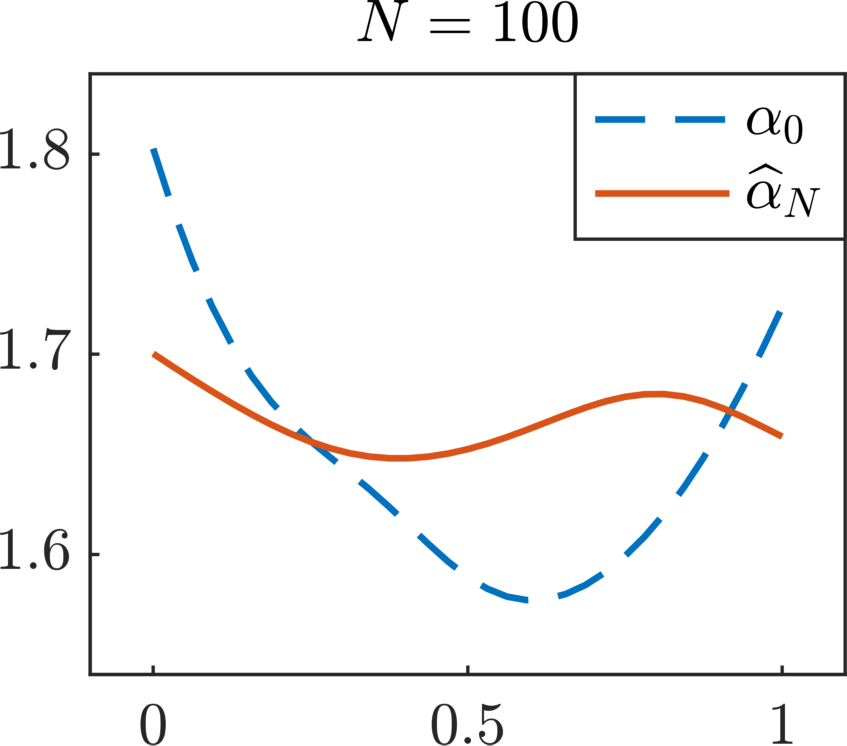} \\[0.5cm]
\includegraphics{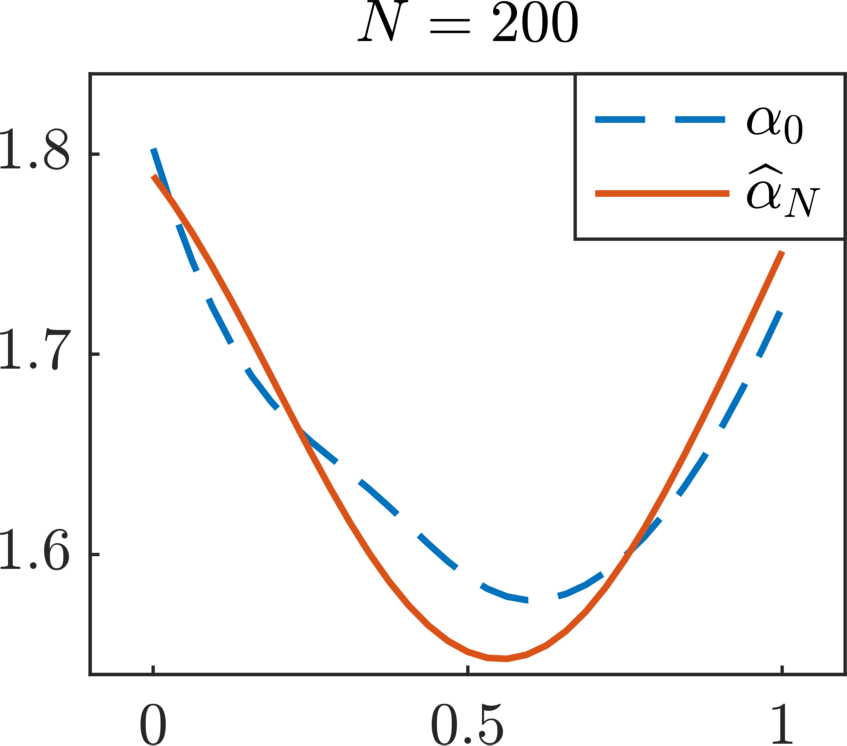} &
\includegraphics{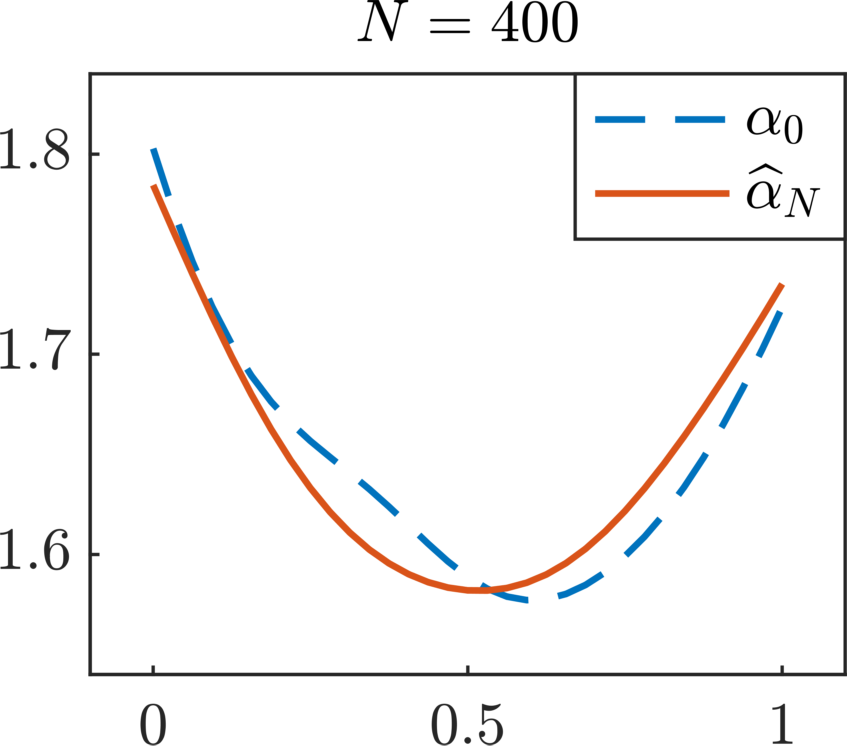} &
\includegraphics{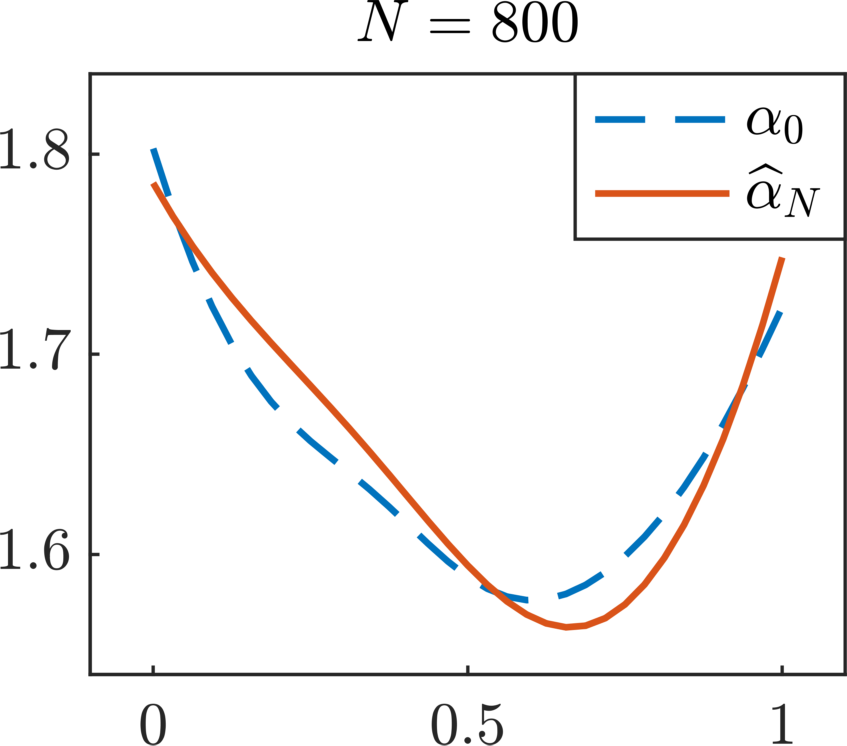}
\end{tabular}
\end{center}
\caption{Comparison of the true 1D permittivity field $\alpha_0$ and the posterior mean estimator $\widehat\alpha_N$ for varying numbers of observations $N \in \{25, 50, 100, 200, 400, 800\}$. The reconstruction accuracy improves as the number of observations increases.}
\label{fig:alpha_ex_hat_1D_N}
\end{figure}

\begin{figure}
\begin{center}
\includegraphics{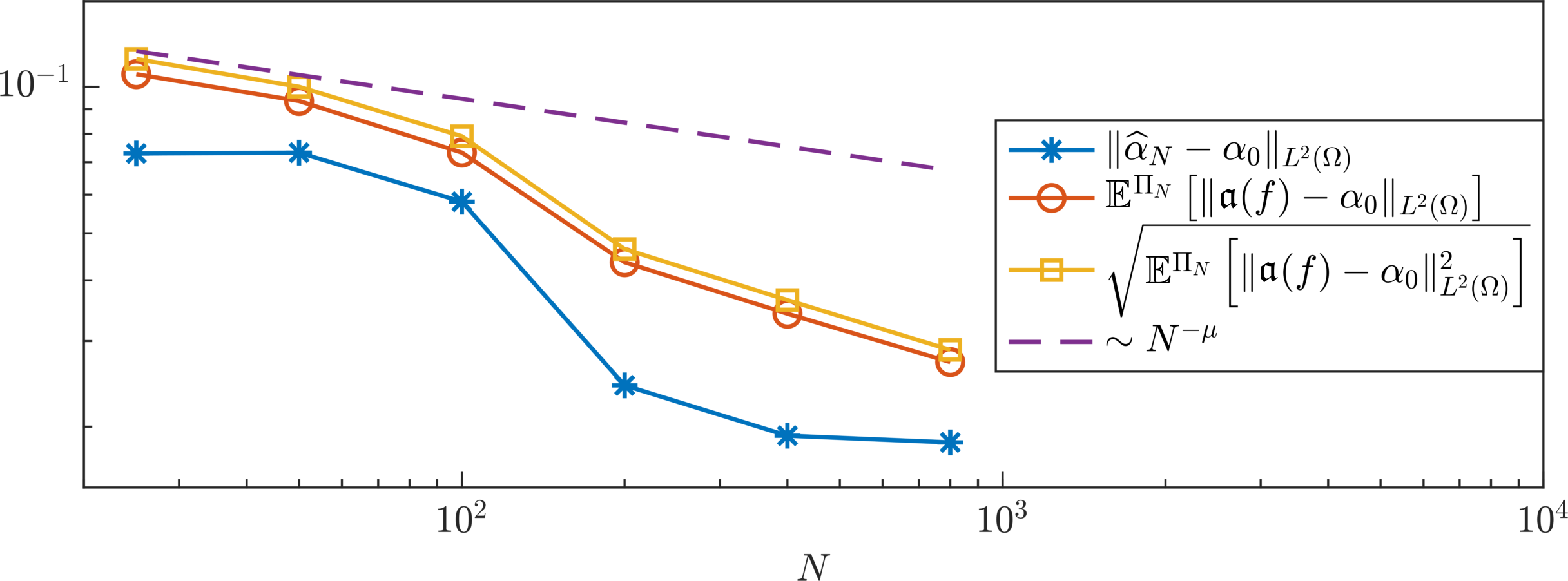}
\end{center}
\caption{Convergence of the approximation errors in the 1D setting. The plot compares the $L^2$ error of the posterior mean, the mean absolute error (MAE), and the mean squared error (MSE) as functions of the number of observations $N$, together with the theoretical convergence rate $N^{-\mu}$.}
\label{fig:alpha_ex_hat_1D_rateN}
\end{figure}

\begin{figure}
\begin{center}
\begin{tabular}{ccccc}
\includegraphics{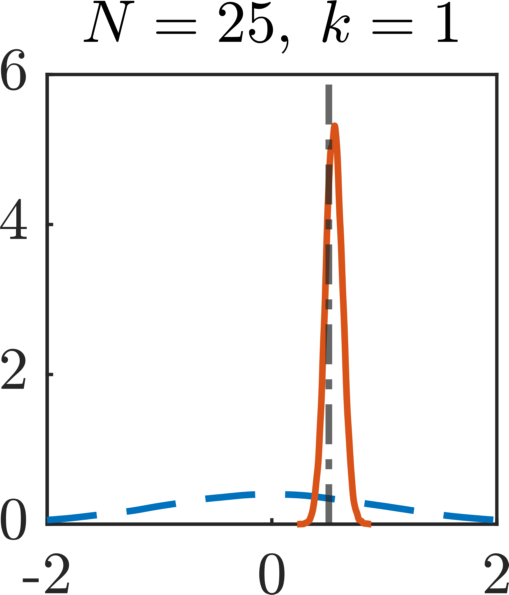} &
\includegraphics{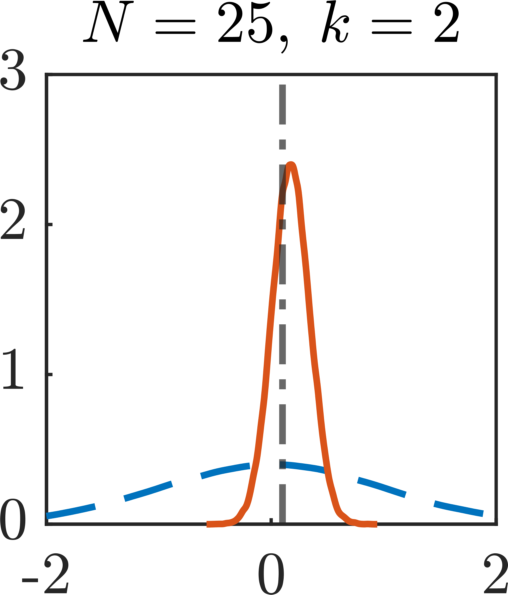} &
\includegraphics{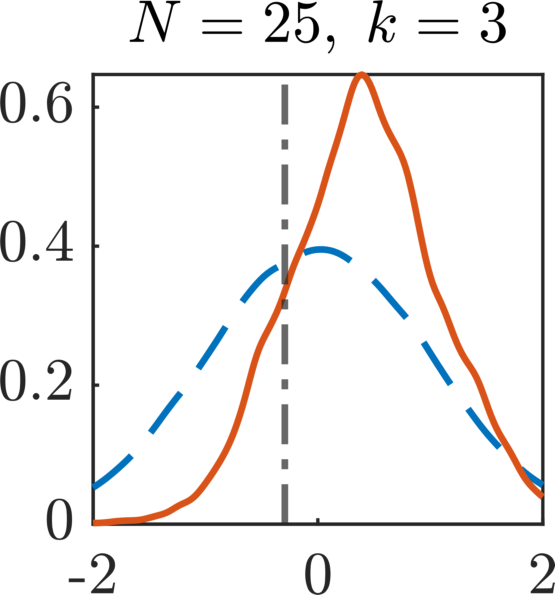} &
\includegraphics{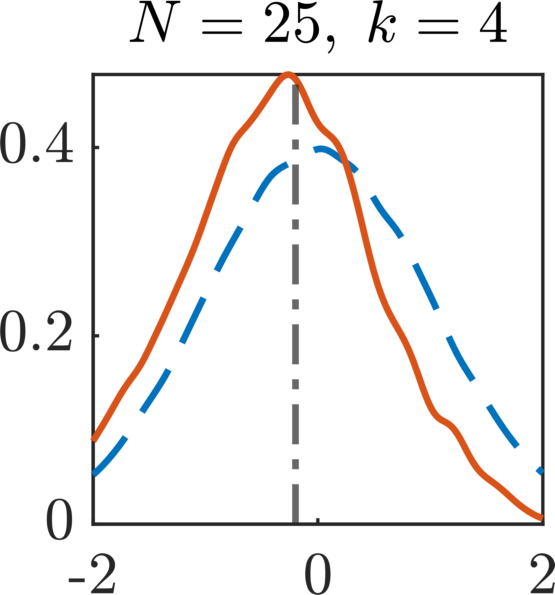} &
\includegraphics{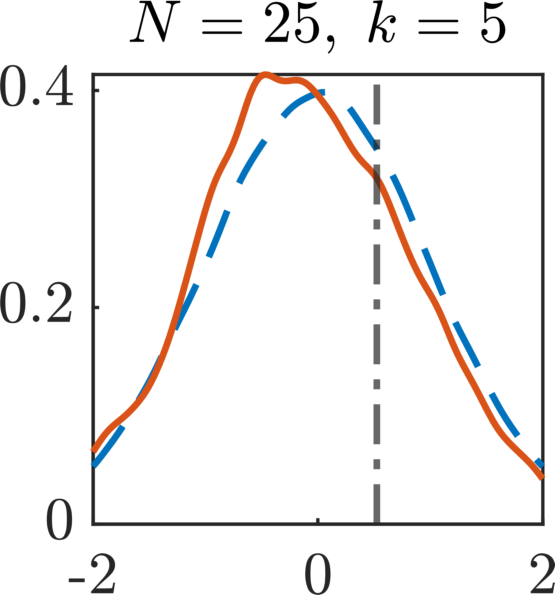} \\
\includegraphics{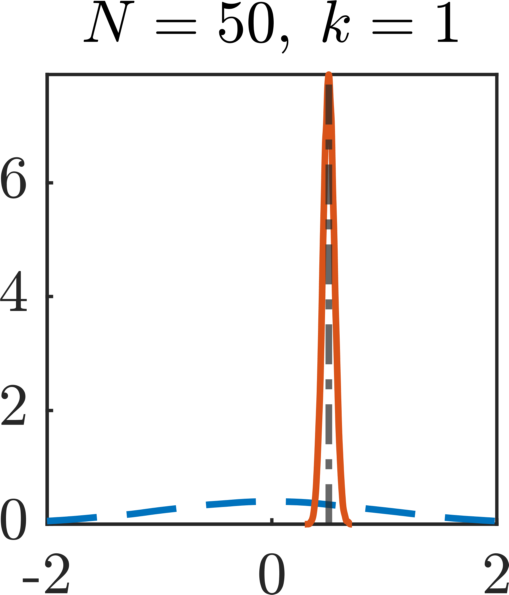} &
\includegraphics{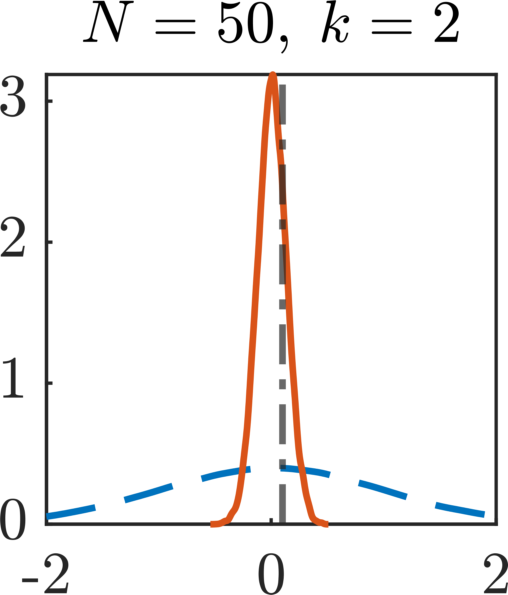} &
\includegraphics{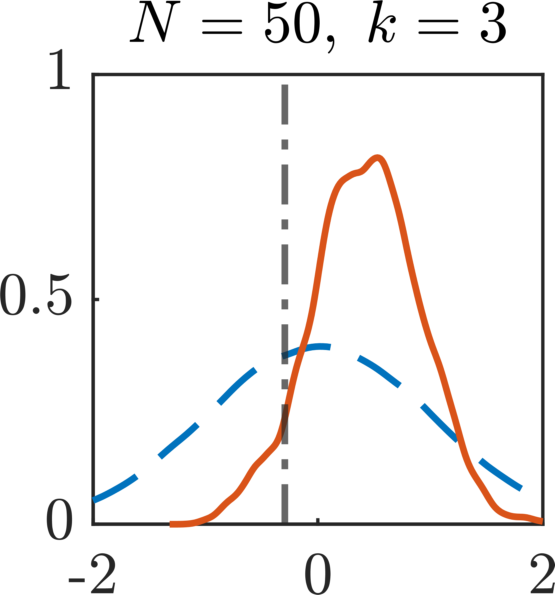} &
\includegraphics{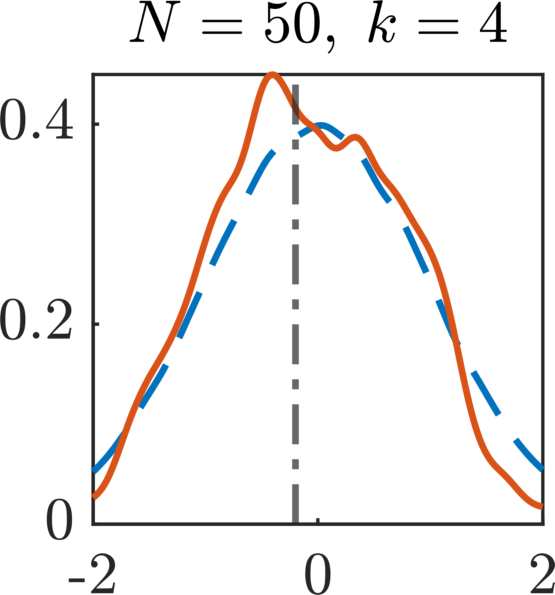} &
\includegraphics{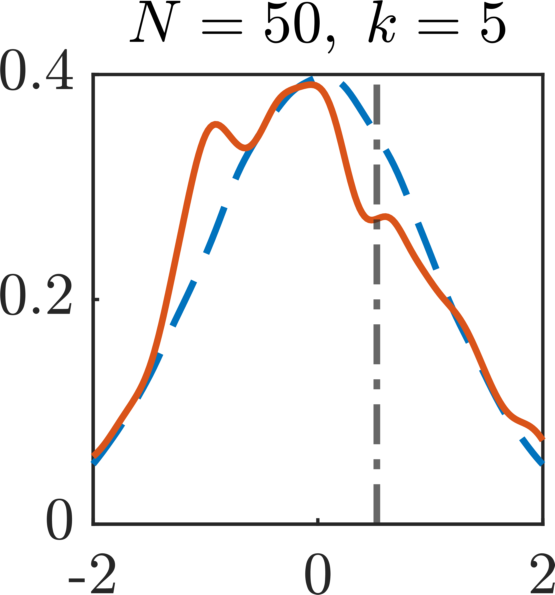} \\
\includegraphics{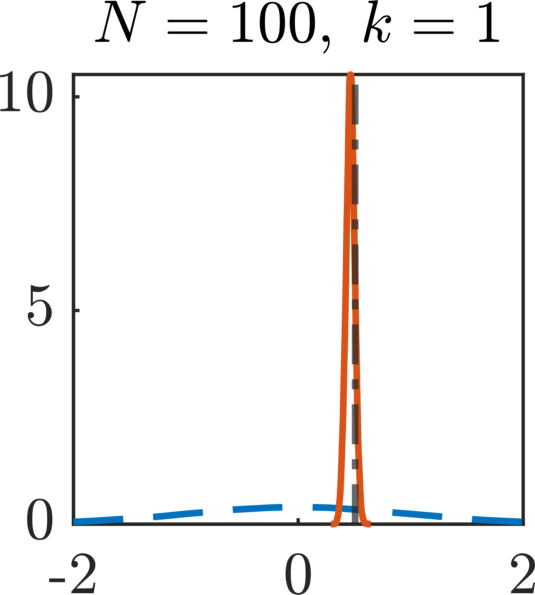} &
\includegraphics{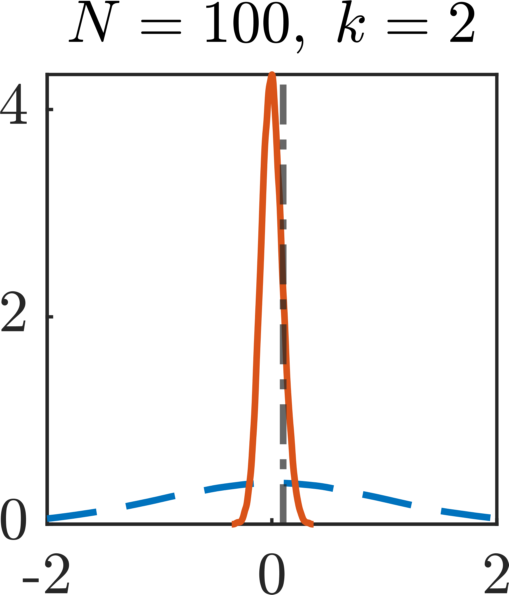} &
\includegraphics{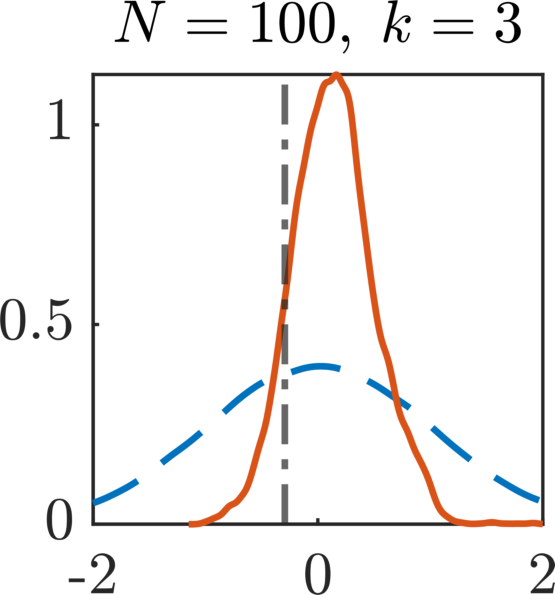} &
\includegraphics{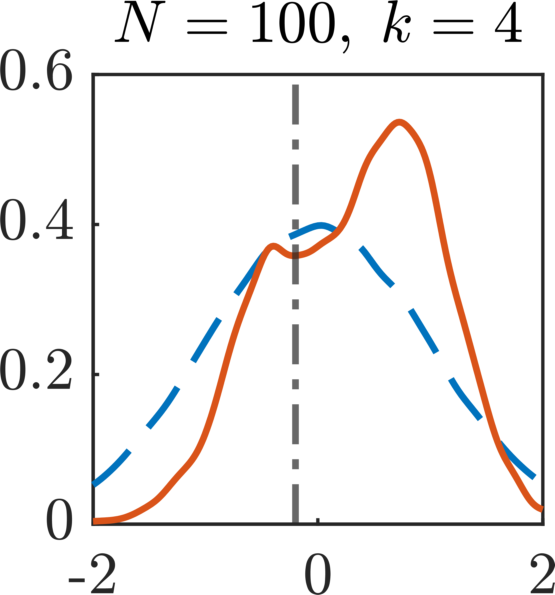} &
\includegraphics{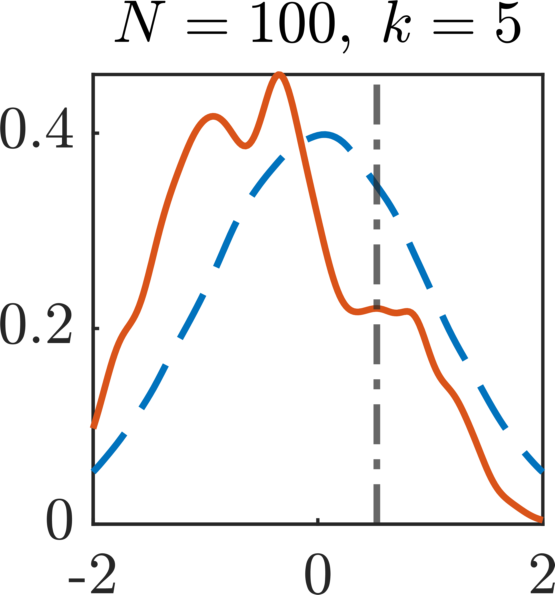} \\
\includegraphics{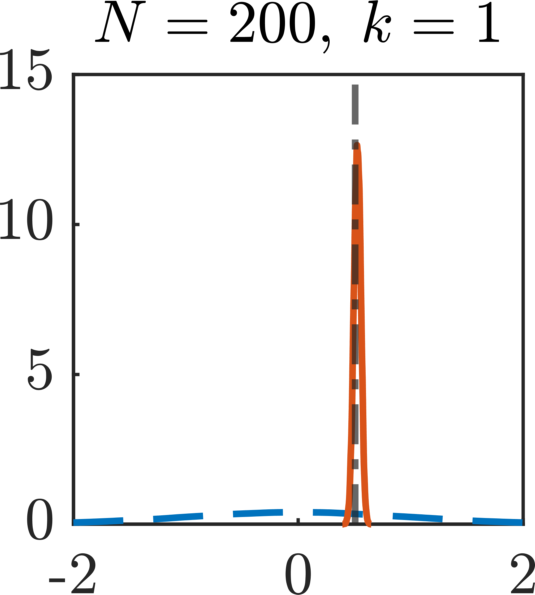} &
\includegraphics{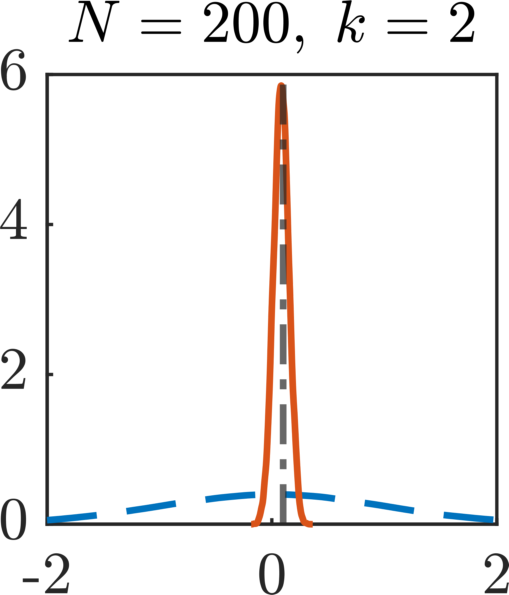} &
\includegraphics{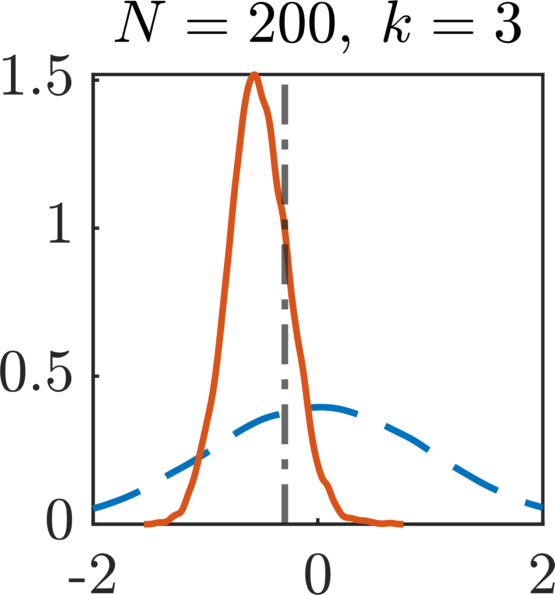} &
\includegraphics{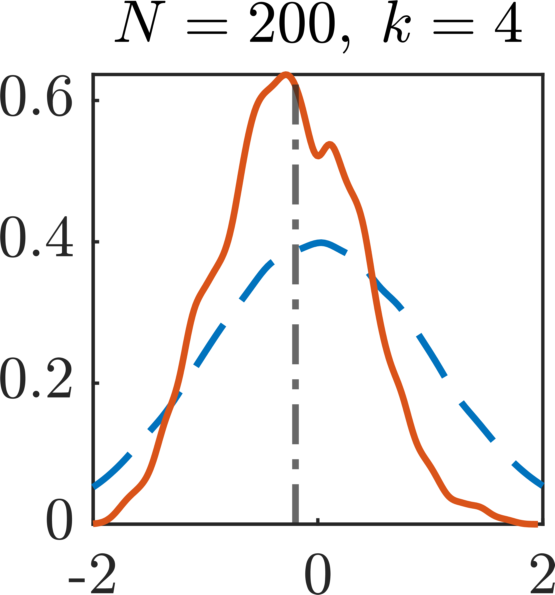} &
\includegraphics{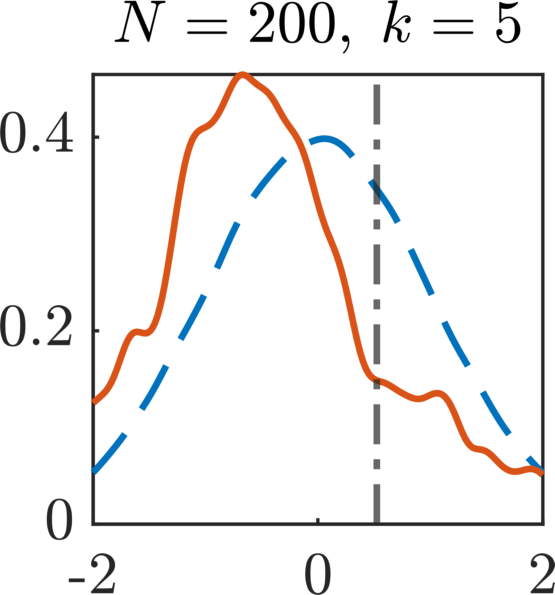} \\
\includegraphics{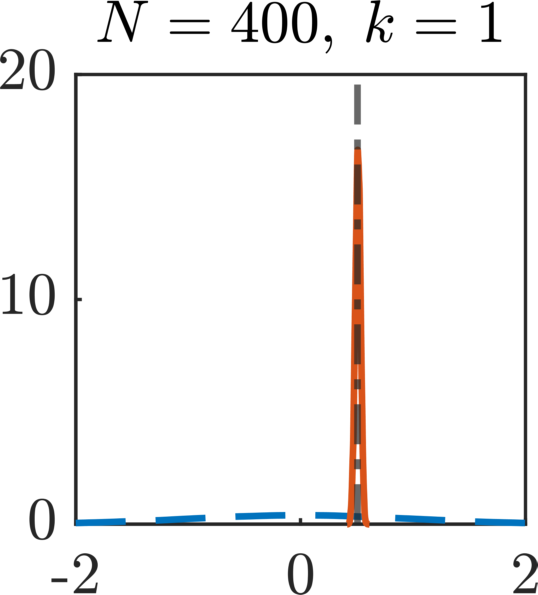} &
\includegraphics{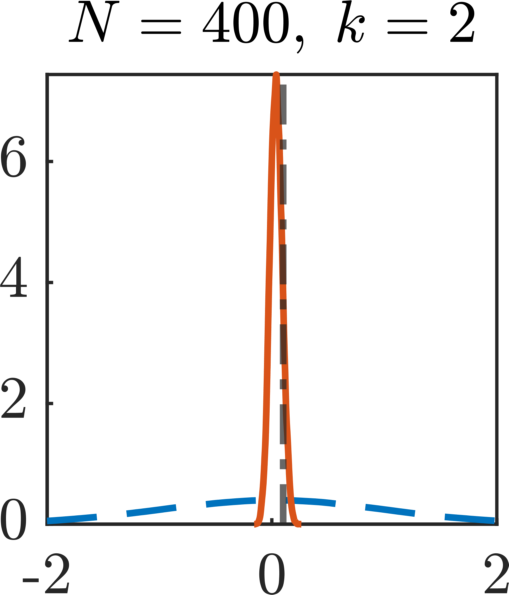} &
\includegraphics{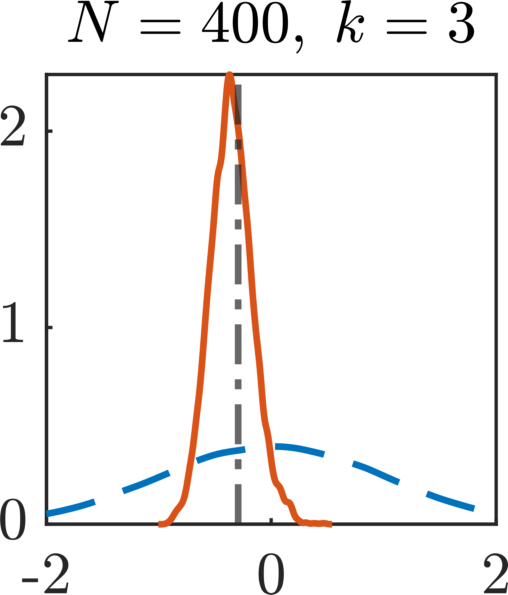} &
\includegraphics{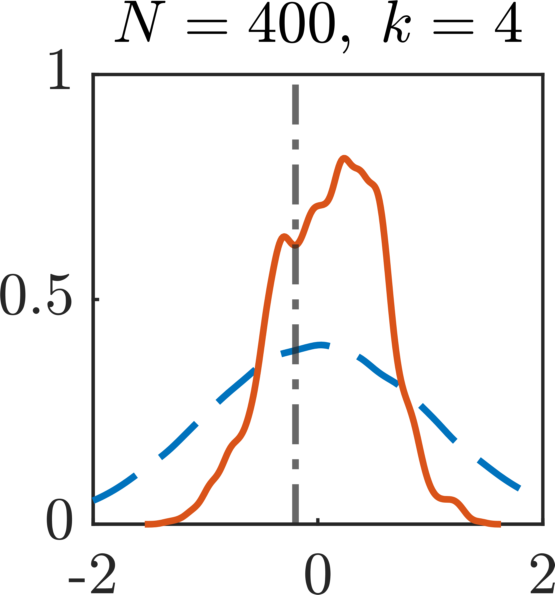} &
\includegraphics{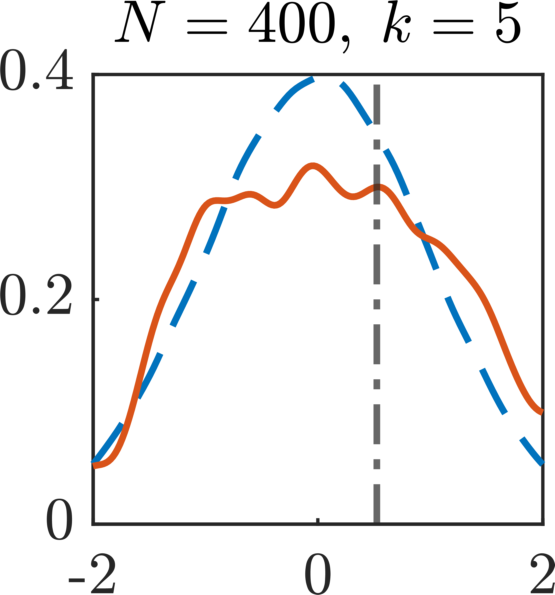} \\
\includegraphics{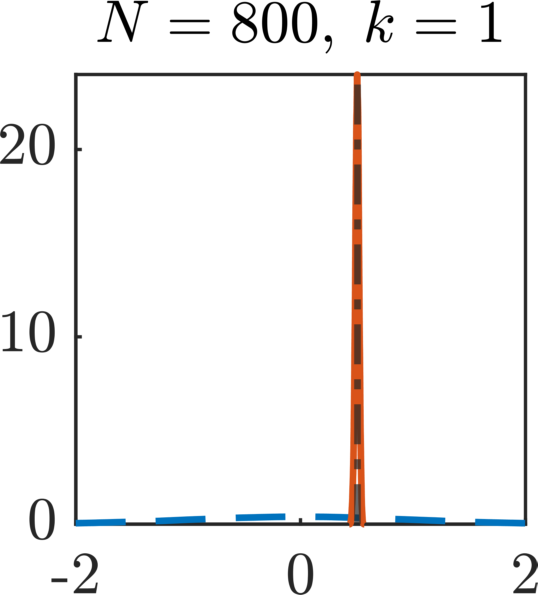} &
\includegraphics{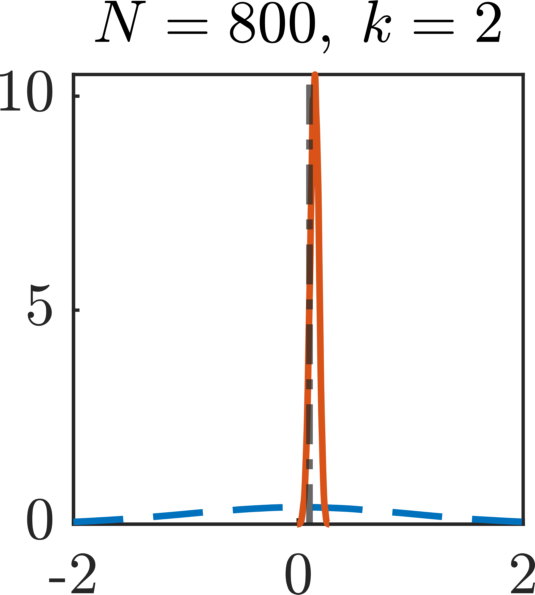} &
\includegraphics{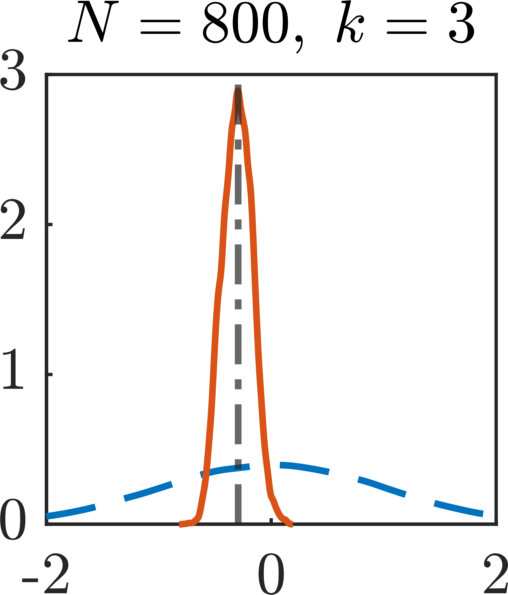} &
\includegraphics{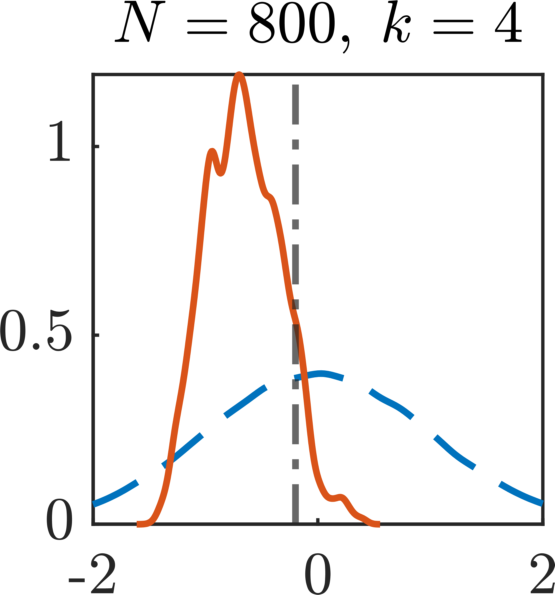} &
\includegraphics{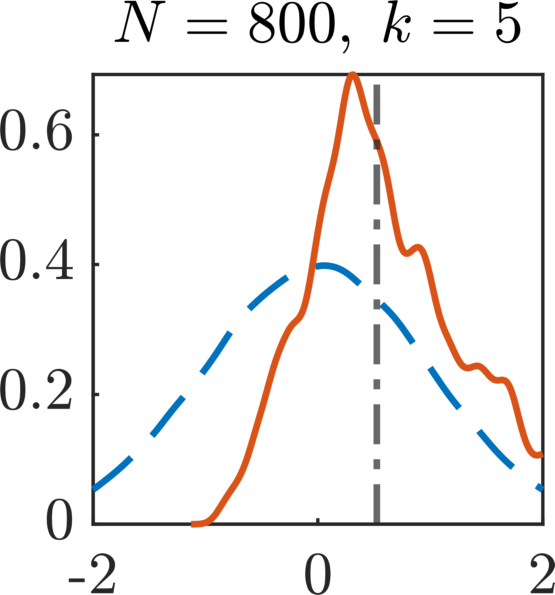}
\end{tabular}
\end{center}
\caption{Marginal distributions projected onto the first five Karhunen--Loève modes in the 1D setting, for increasing values of $N$ (top to bottom). The dashed blue and solid orange lines represent the prior and posterior distributions, respectively, while the vertical dash-dotted lines represent the true coefficients.} 
\label{fig:projections_1D_pN}
\end{figure}

We first consider the one-dimensional setting ($d = 1$) on the spatial domain $[0,L]$ with $L=1$. The initial density in equation \eqref{eq:PDE_system_IC} is given by
\begin{equation}
u_0(x) = \frac{1 + 10 x (L - x)}{L + \frac{5}{3} L^3}.
\end{equation}
Both the spatial domain $[0,L]$ and the temporal domain $[0,T]$ are partitioned into $32$ elements. The observations $\mathcal D_N$ are sampled uniformly in the spatio-temporal cylinder $\Omega \times [t_0,t_1]$, with $t_0 = 0.1$ and $t_1 = 0.9$. We set the hyperparameters of the covariance \eqref{eq:covariance} to $\varsigma = 10^4$, $\tau = 5$, and $\theta = 5.6$, so that \cref{as:prior} is satisfied.

To empirically investigate the theoretical contraction results of \cref{thm:contraction_parameter}, we solve the inverse problem for varying numbers of observations $N \in \{25, 50, 100, 200, 400, 800\}$. The proposal step size $\zeta$ of the pCN algorithm is tuned to yield an acceptance rate between $10\%$ and $40\%$. Specifically, for the respective values of $N$, we set $\zeta \in \{0.1, 0.06, 0.05, 0.04, 0.03, 0.03\}$. We extract $10^5$ samples and discard the first $10^4$ as burn-in.

In \cref{fig:alpha_ex_hat_1D_N}, we compare the posterior mean estimator $\widehat\alpha_N$ with the true permittivity field. As expected, a small number of observations is insufficient to accurately reconstruct the exact coefficient. However, as $N$ increases, the posterior mean estimator successfully captures the true field. This behavior is further illustrated in \cref{fig:alpha_ex_hat_1D_rateN}, which displays the $L^2$ error of the posterior mean estimator together with the mean absolute error (MAE) and mean squared error (MSE), compared with the theoretical convergence rate $N^{-\mu}$. Here, $\mu$ is computed by setting $d = 1$, $s = \theta$, and $\kappa = \ell - 1$ with $\ell = 4$, corresponding to the assumption that the optimal regularity in \cref{rem:optimal_regularity} is attained, yielding $\mu \simeq 0.153$.

The empirical rates obtained from the MAE and MSE appear slightly faster than the theoretical bound, whereas the $L^2$ error of the posterior mean exhibits a more variable decay. We emphasize that the convergence rate in \cref{thm:contraction_parameter} bounds the posterior contraction in probability with respect to the data distribution $P_N$, rather than holding almost surely. Thus, the $L^2$ error of a single posterior mean estimate is not guaranteed to strictly follow this rate. Furthermore, the MAE and MSE are integrated with respect to the posterior $\Pi_N$, offering a related but distinct metric from the exact theoretical bound. Consequently, \cref{fig:alpha_ex_hat_1D_rateN} serves primarily as a qualitative validation, demonstrating error reduction as data increases, rather than a strict quantitative verification of the convergence rate.

Finally, in \cref{fig:projections_1D_pN}, we show the marginal posterior distributions projected onto the first five modes of the Karhunen--Loève expansion for each value of $N$. The distributions concentrate around the exact values of the coefficients. For the first three modes, the distributions shrink, the peaks become taller, and the bias decreases as the number of observations increases. The higher modes initially resemble the prior distribution and start changing shape when $N$ grows. However, even the largest dataset ($N = 800$) is insufficient to reliably capture the fourth and fifth modes.

\subsection{Two-dimensional example} \label{ssec:2D}

\begin{figure}
\begin{center}
\begin{tabular}{cccc}
\includegraphics{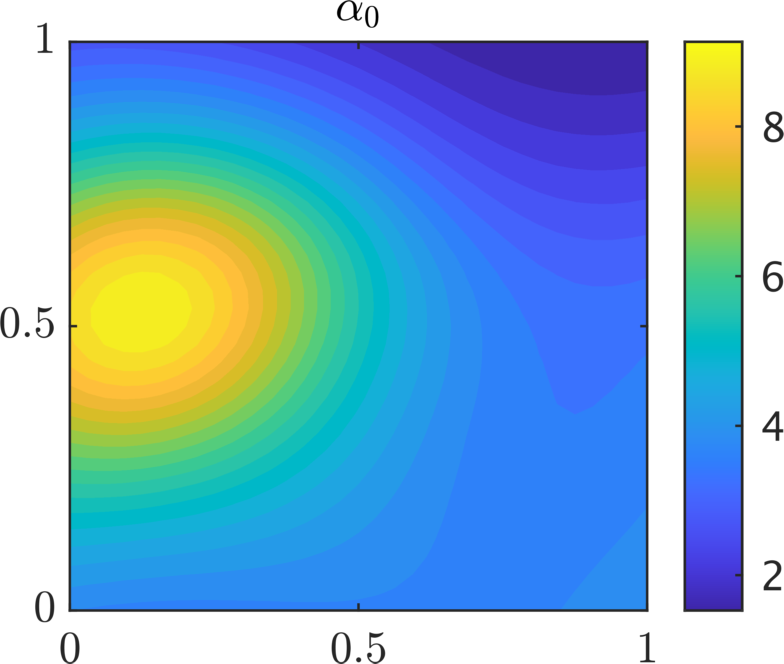} &&
\includegraphics{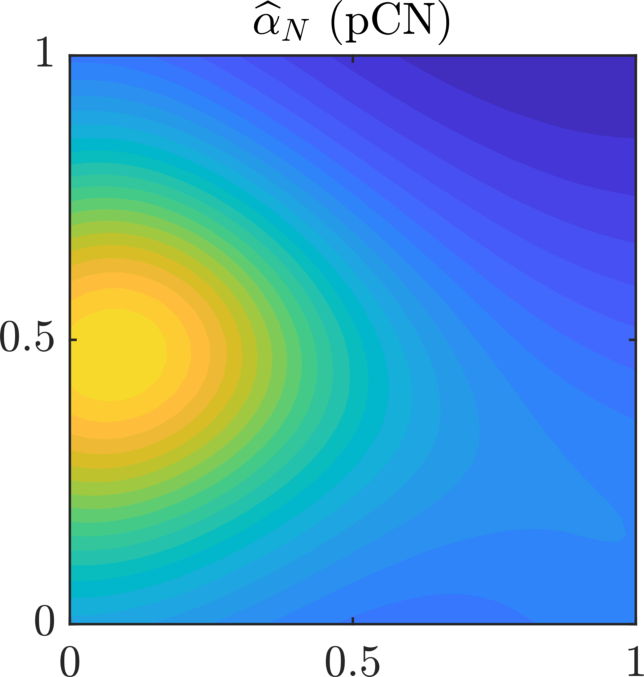} &
\includegraphics{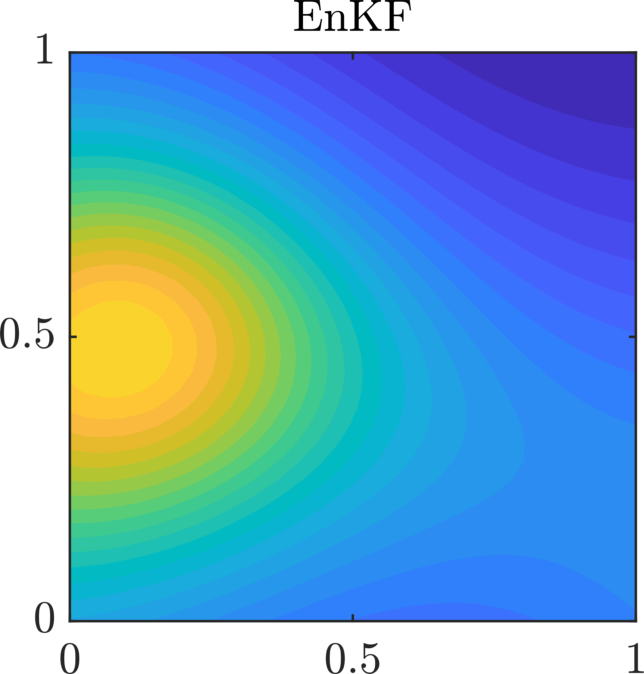}
\end{tabular}
\end{center}
\caption{Comparison in the 2D setting for $N=500$ observations. Left: true permittivity field $\alpha_0$. Center: posterior mean obtained via the pCN algorithm. Right: ensemble mean obtained at the second and final iteration of the EnKF.}
\label{fig:alpha_ex_hat_2D}
\end{figure}

\begin{figure}
\begin{center}
\begin{tabular}{ccccc}
\includegraphics{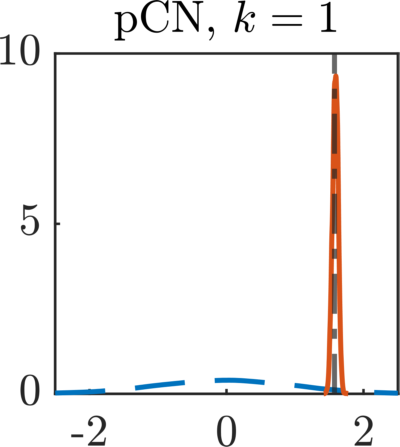} &
\includegraphics{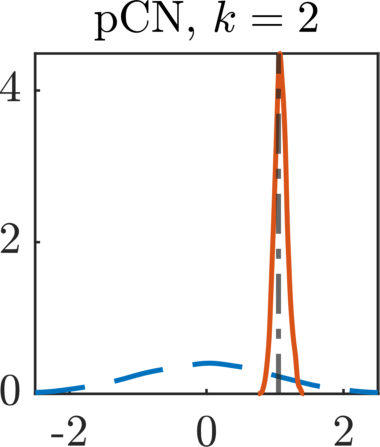} &
\includegraphics{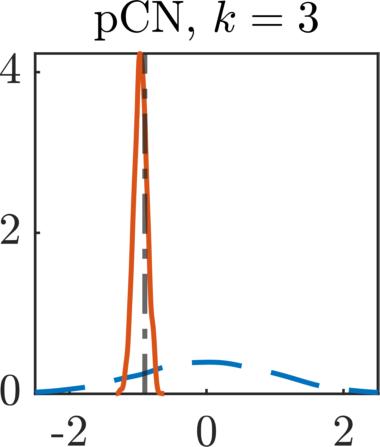} &
\includegraphics{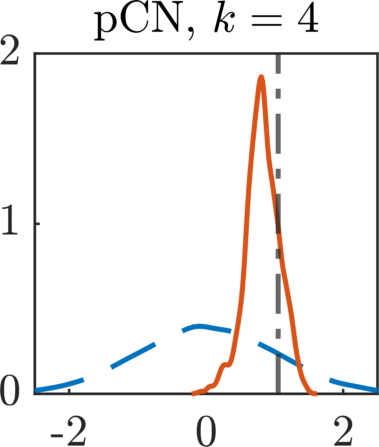} &
\includegraphics{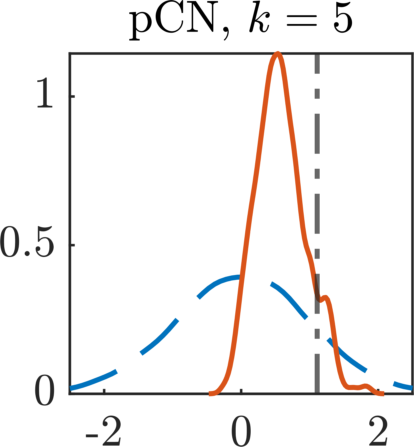} \\
\includegraphics{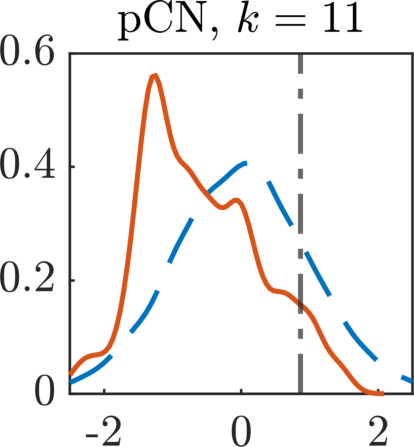} &
\includegraphics{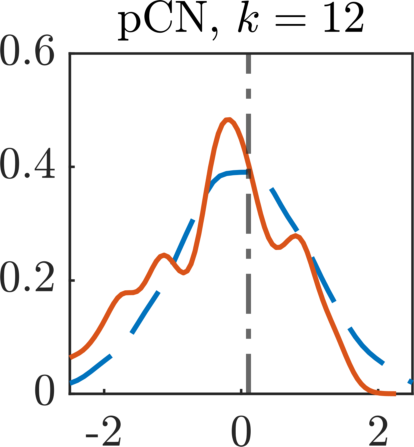} &
\includegraphics{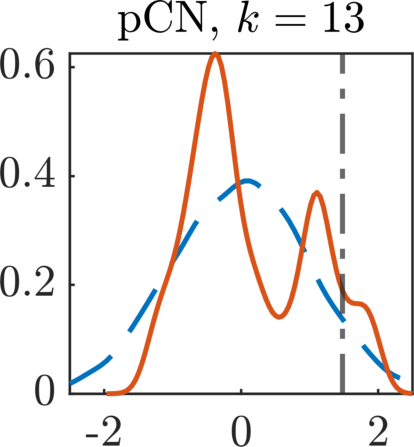} &
\includegraphics{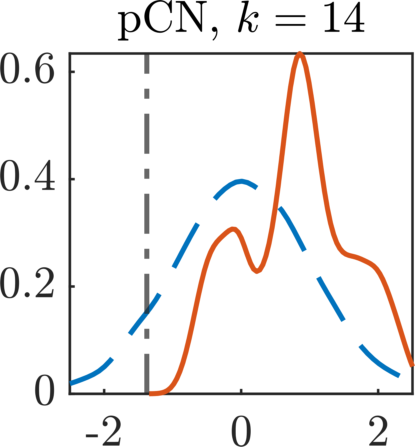} &
\includegraphics{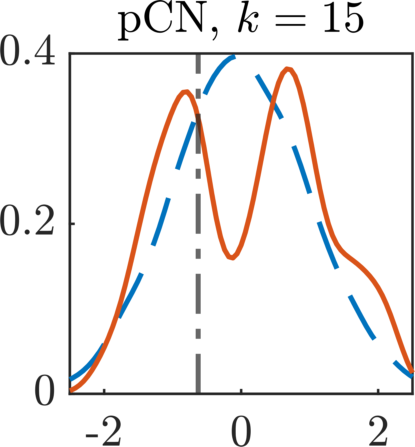} \\[1cm]
\includegraphics{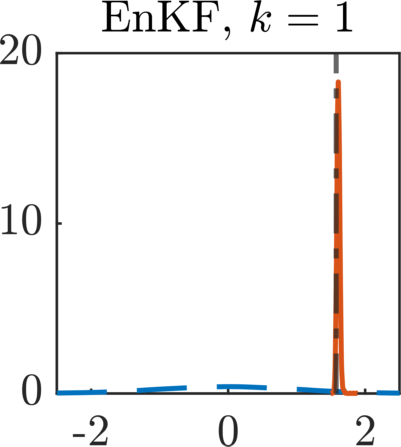} &
\includegraphics{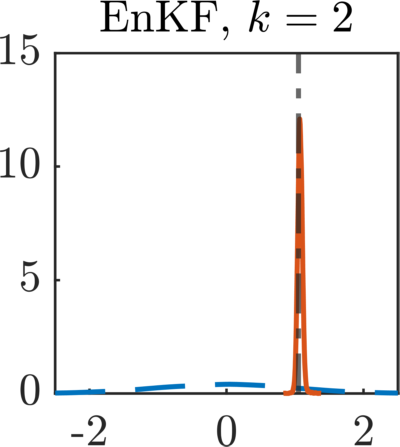} &
\includegraphics{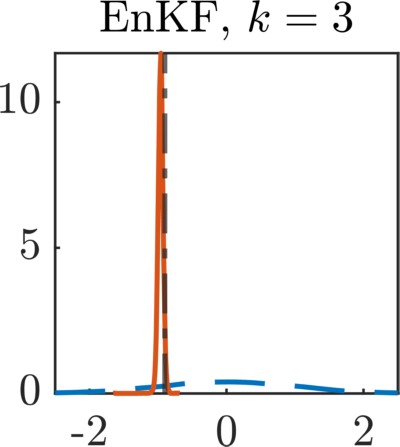} &
\includegraphics{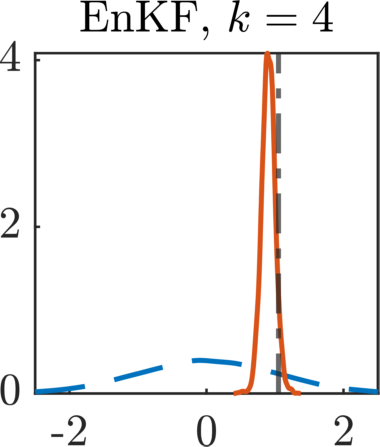} &
\includegraphics{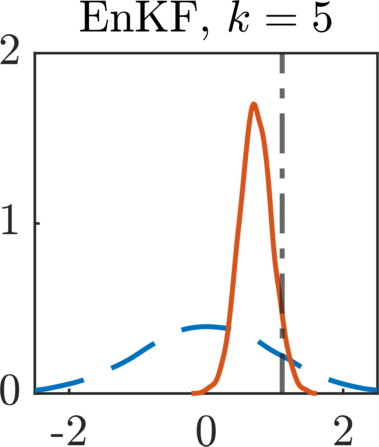} \\
\includegraphics{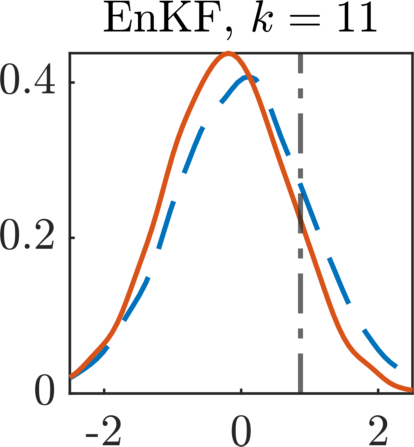} &
\includegraphics{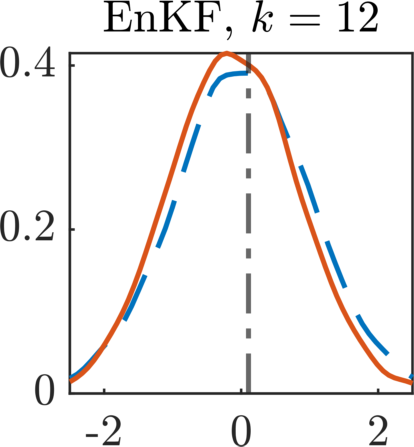} &
\includegraphics{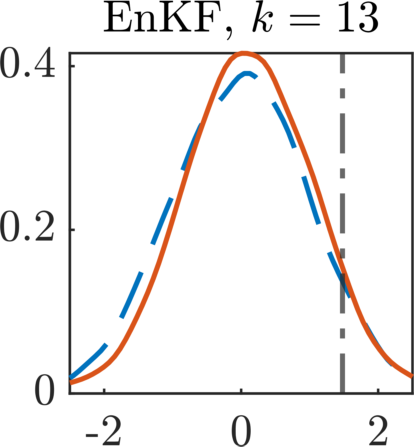} &
\includegraphics{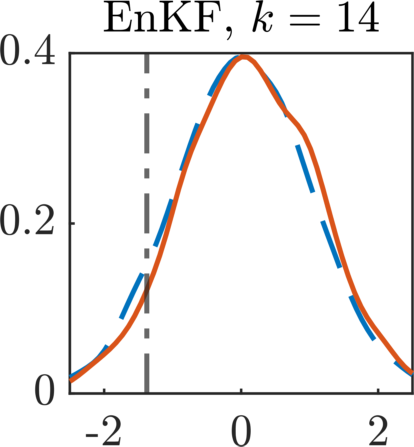} &
\includegraphics{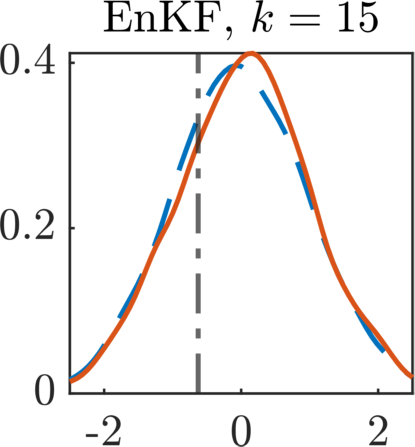}
\end{tabular}
\end{center}
\caption{Marginal distributions for several Karhunen--Loève modes in the 2D setting with $N=500$. The top two rows display the results from the pCN algorithm, while the bottom two rows display the results from the EnKF. The dashed blue and solid orange lines represent the prior and posterior distributions, respectively, while the vertical dash-dotted lines represent the true coefficients.}
\label{fig:projections_2D_p}
\end{figure}

We now consider the inverse problem in the two-dimensional case ($d=2$). We compare the performance of the pCN algorithm with that of the Ensemble Kalman Filter (EnKF) \cite{ILS13} implemented via the \texttt{EnsembleKalmanProcesses} package. The spatial domain is $[0,L]^2$ with $L=1$, and the initial density in equation \eqref{eq:PDE_system_IC} is given by
\begin{equation}
u_0(x) = \frac{\left( 1 + 10 x_1 (L - x_1) \right) \left( 1 + 10 x_2 (L - x_2) \right)}{\left( L + \frac{5}{3} L^3 \right)^2}.
\end{equation}
Notice that this is an example in which the compatibility condition between the initial and boundary conditions is not satisfied. In fact, by \cref{rem:compatibility_conditions}, the compatibility condition requires $\partial u_0 / \partial\nu > 0$ everywhere on $\partial\Omega$, which fails, for example, on the boundary face $x_1 = 0$. Therefore, this illustrates that the restriction $t_0 > 0$ is not merely a technicality. The spatial domain $[0,L]^2$ is partitioned into $32$ elements per dimension, and the temporal domain $[0,T]$ is similarly divided into $32$ steps. The $N = 500$ observations $\mathcal D_N$ are sampled uniformly in the spatio-temporal domain $\Omega \times [t_0,t_1]$ with $t_0 = 0.2$ and $t_1 = 0.8$. We set the hyperparameters of the covariance \eqref{eq:covariance} to $\varsigma = 5 \cdot 10^4$, $\tau = 5$, and $\theta = 6.5$, so that \cref{as:prior} is satisfied.

In \cref{fig:alpha_ex_hat_2D}, we compare the posterior mean estimator $\widehat\alpha_N$ with the ensemble mean at the second iteration of the EnKF. We generate $10^4$ samples with the pCN algorithm and discard the first $10^3$ as burn-in, using a step size of $\zeta=0.05$. The number of particles in the ensemble of the EnKF is set to $5 \cdot 10^3$ to ensure the computational budget matches that of the pCN algorithm. Qualitatively, despite the limited number of observations, both methods successfully recover the main features of the parameter field at a reasonable computational cost.

Moreover, from the marginal posterior distributions projected onto the modes of the Karhunen--Loève expansion, we observe that the EnKF provides a slightly more accurate approximation of the true coefficients. Similar to the one-dimensional test case, the distributions contract around the true coefficients for the lower modes (1--5), while they remain closer to the prior distribution for the higher modes (11--15), indicating that the data we employ do not provide sufficient information to reliably infer these coefficients.

\section{Conclusion} \label{sec:conclusion}

We have studied the nonparametric inverse problem of recovering the permittivity of a coupled Fokker--Planck--Darcy system from discrete observations of the density, under the physically natural no-flux boundary conditions. On the analytical side, we established well-posedness, together with uniform upper and lower bounds for the solution of the coupled system. On the statistical side, we proved forward and backward stability estimates and derived a posterior contraction rate, together with the corresponding rate for the posterior mean estimator. Numerical experiments in one and two dimensions confirm the qualitative behavior predicted by the theory.

Several questions remain open. The most immediate concerns the regularity exponent $\kappa$ in \cref{as:regularity}, which is primarily used in the proof of \cref{cor:backward_stability_estimate}. As discussed in \cref{rem:optimal_regularity}, the argument there yields the assumption only for $\kappa<3$, whereas the natural expectation is $\kappa=\ell-1$, the highest regularity compatible with that of the permittivity. Establishing this bound, with constants uniform over the admissible set, would immediately improve the exponent $\mu$ in \cref{thm:contraction_parameter}, yielding the optimal rate obtainable within the present approach. It would therefore be desirable to prove the bounds in \cref{as:regularity} rather than assume them. A second question concerns the global solvability of the forward problem, which, to our knowledge, is already open for constant permittivity in three dimensions. Extending the two-dimensional result of \cite{BWN94} to a space-dependent permittivity would allow the time horizon to be taken arbitrarily large in dimensions $d\le2$. This would involve studying the free energy of the system and is related to the problem of inferring the unknown permittivity from discrete observations of the stationary solution, i.e., in the large-time regime. A third question concerns the constants. As recorded in \cref{rem:K_large}, the stability constant produced by the multiplier argument of \cite{Wan26} depends doubly exponentially on the final time, through the lower bound on the density, so that the resulting estimate is qualitative rather than quantitative. A stability constant with polynomial dependence on $T$ would make the theory considerably more informative in practice. Moreover, after establishing global-in-time solvability and combining it with the stability estimates, it would be interesting to determine whether the constants can be chosen independently of the observation time altogether.

Finally, there are several natural extensions. The two-species Debye system of \cite{BWN94}, in which two densities of opposite charge are coupled through the same potential, leads to an inverse problem of the same type and would be a natural next step. It would also be of interest to establish minimax lower bounds to assess the optimality of the rate obtained here, to treat continuous-time or trajectory-based observation schemes, and to analyze the ensemble Kalman filter, which, as suggested by our two-dimensional experiments, performs comparably to posterior sampling, but for which no contraction theory is presently available in this setting.

\paragraph*{Acknowledgements.}

G.A.P. is partially supported by an ERC-EPSRC Frontier Research Guarantee through Grant No. EP$/$X038645, ERC Advanced Grant No. 247031. A.M.S. acknowledges funding support from a USA Department of Defense Vannevar Bush Faculty Fellowship (award N00014-22-1-2790). A.Z. is supported by ``Centro di Ricerca Matematica Ennio De Giorgi'' and the ``Emma e Giovanni Sansone'' Foundation, and is a member of INdAM-GNCS. The authors thank Oliver R. A. Dunbar for his help with the \texttt{EnsembleKalmanProcesses} Julia package.

\paragraph*{AI acknowledgement.}

Claude (Anthropic) was used during the preparation of this manuscript. In particular, it was used to check the proofs of~\cref{pro:boundedness_u,pro:boundedness_u_below,lem:entropy_estimate,pro:Lipschitz,pro:backward_stability_estimate,cor:backward_stability_estimate}, it identified errors in earlier drafts and suggested corrections. It proposed the truncation argument used in the proof of \cref{pro:boundedness_u_below}. It was also used to verify the numerical parameters used in \cref{sec:numerics} and it assisted with the use of the \texttt{EnsembleKalmanProcesses} and \texttt{Gridap} packages in Julia. All mathematical statements, proofs, and numerical results were independently verified by the authors, who take full responsibility for them.

%\appendix
%\input{main_appendix.tex}

\newpage

\bibliographystyle{siamnodash}
\bibliography{bibliography.bib}

\end{document}